\documentclass[12pt,reqno]{amsart}
\usepackage[utf8]{inputenc}

\title[Promotion and RSK in moon polyominoes]{Piecewise-linear promotion and RSK in rectangles and moon polyominoes}
\author{Joseph Johnson}
\address{Joseph Johnson, Department of Mathematics, North Carolina State University, Raleigh, NC 27695}
\email{jwjohns5@ncsu.edu}

\author{Ricky Ini Liu}
\address{Ricky Ini Liu, Department of Mathematics, University of Washington, Seattle, WA 98195}
\email{riliu@uw.edu}
\thanks{The second author was partially supported by grants from the National Science Foundation (DMS-1700302/2204415 and CCF-1900460).}
\date{\today}

\usepackage{amssymb}
\usepackage{amsfonts}
\usepackage{amsthm}
\usepackage{amsmath}
\usepackage{amscd}
\usepackage{enumerate}
\usepackage{t1enc}
\usepackage[mathscr]{eucal}
\usepackage{indentfirst}
\usepackage{graphicx}
\usepackage{graphics}
\usepackage{hyperref}
\usepackage{pict2e}
\usepackage{epic}
\numberwithin{equation}{section}
\usepackage[margin=2.9cm]{geometry}
\usepackage{epstopdf} 
\usepackage{transparent}
\usepackage{wasysym}

\usepackage{caption}

\usepackage{tikz-cd}
\usepackage{tikz}
\usetikzlibrary{decorations.pathreplacing,calligraphy}

\usepackage{mathdots}
\usepackage{todonotes}
\usepackage{verbatim}

 \colorlet{myGreen}{green!50!gray!120!}

\theoremstyle{plain}
\newtheorem{Th}{Theorem}[section]
\newtheorem{Lemma}[Th]{Lemma}
\newtheorem{Cor}[Th]{Corollary}
\newtheorem{Prop}[Th]{Proposition}

 \theoremstyle{definition}
\newtheorem{Def}[Th]{Definition}

\newtheorem{?}[Th]{Problem}
\newtheorem{Ex}[Th]{Example}
\newtheorem{Obs}[Th]{Observation}

\newcommand{\op}{\mathcal{O}(P)}
\newcommand{\cp}{\mathcal{C}(P)}

\DeclareMathOperator{\Pro}{Pro}
\DeclareMathOperator{\evac}{Evac}
\DeclareMathOperator{\RSK}{RSK}
\DeclareMathOperator{\swpro}{SWPro}

\DeclareMathOperator{\qstab}{QSTAB}
\DeclareMathOperator{\Ehr}{Ehr}

\newcommand{\proid}{\Pro^P}
\newcommand{\idpro}{\Pro^Q}

\newcommand{\ee}{\evac^P \circ \evac^Q}
\newcommand{\idproinv}{\left(\Pro^Q\right)^{-1}}
\newcommand{\ei}{\evac^P}
\newcommand{\ie}{\evac^Q}

\newcommand{\m}{\mathcal{M}}
\newcommand{\n}{\mathcal{N}}

\newcommand{\qm}{\qstab(\mathcal{M})}
\newcommand{\qn}{\qstab(\mathcal{N})}

\newcommand{\concat}{\oslash}

\tikzstyle{wB}=[circle, draw, fill=black, inner sep=0pt, minimum width=4.5pt]
\tikzstyle{wR}=[circle, draw, fill=red, inner sep=0pt, minimum width=4.5pt]
\tikzstyle{wBlue}=[circle, draw, fill=blue, inner sep=0pt, minimum width=4.5pt]

\tikzstyle{bigB}=[circle, draw, fill=black, inner sep=0pt, minimum width=6pt]
\tikzstyle{bigR}=[circle, draw, fill=red, inner sep=0pt, minimum width=6pt]
\tikzstyle{bigBlue}=[circle, draw, fill=blue, inner sep=0pt, minimum width=6pt]

\begin{document}

\maketitle

\begin{abstract} 
We study piecewise-linear and birational lifts of Sch\"utzenberger promotion, evacuation, and the RSK correspondence defined in terms of toggles. Using this perspective, we prove that certain chain statistics in rectangles shift predictably under the action of these maps. We then use this to construct piecewise-linear and birational versions of Rubey's bijections between fillings of equivalent moon polyominoes that preserve these chain statistics, and we show that these maps form a commuting diagram. We also discuss how these results imply Ehrhart equivalence and Ehrhart quasi-polynomial period collapse of certain analogues of chain polytopes for moon polyominoes.
\end{abstract}

\section{Introduction}

For any finite poset $P$, Cameron and Fon-Der-Flaass \cite{cameronfonderflaass}  introduce an action $\rho$ on the order ideals of $P$ called \emph{combinatorial rowmotion} that sends $I$ to the order ideal generated by the minimal elements of $P \setminus I$. Rowmotion has been studied by many authors due to its desirable dynamical properties on certain graded posets \cite{cameronfonderflaass,fonderflaass,grinbergroby1,grinbergroby2,josephroby1,josephroby2,propproby,strikerwilliams,thomaswilliams}. For example, on the product of two chains $R=[r] \times [s]$ (called the \emph{rectangle poset}) the order of rowmotion is $r+s$ \cite{cameronfonderflaass}. The study of rowmotion has also led to a number of results about homomesy and cyclic sieving phenomena---see \cite{einsteinpropp1,josephroby1,musikerroby,propproby,thomaswilliams}.

Rowmotion can also be described as a composition of local involutive transformations called \emph{toggles} \cite{cameronfonderflaass}. Einstein and Propp \cite{einsteinpropp2,einsteinpropp1} encode the order ideals of $P$ as lattice points in $\mathbb{R}_{\geq 0}^{P}$ and define piecewise-linear liftings of combinatorial toggles and rowmotion. They also note that piecewise-linear toggles and rowmotion can be lifted further to birational maps via a procedure called \emph{detropicalization}. (These notions can be studied from a noncommutative perspective as well; see \cite{josephroby2}.) Many properties of rowmotion on the combinatorial level are also true for its piecewise-linear and birational counterparts. For example, the period of piecewise-linear and birational rowmotion on the rectangle poset is still $r+s$ \cite{grinbergroby2}. 

The present authors in \cite[Lemma 4.1]{johnsonliu} proved the following \emph{chain shifting lemma} for rowmotion. Define the \emph{weight} of a path $C$ in a nonnegative labeling $x$ of a rectangle $R$ to be $\sum_{p \in C} x_p$. Then the maximum weight of sets of nonintersecting paths in certain intervals of $R$ shifts downward to a lower interval when we apply $\phi \circ \rho^{-1} \circ \phi^{-1}$, where $\rho$ is (piecewise-linear) rowmotion and $\phi$ is the \emph{transfer map} \cite{stanley2} from the order polytope to the chain polytope. See Figure~\ref{fig:introChainShifting}. (This phenomenon also holds on the noncommutative level for a single path as shown recently in \cite{grinbergroby3}.)

\begin{figure}
    \centering
    \begin{tikzpicture}[scale = 0.75]
    \begin{scope}[rotate=45]
    \draw (0,0) grid (2,3);
    
    \draw [blue, fill = blue, opacity = 0.17] (-0.2,0.8)--(2.2,0.8)--(2.2,3.2)--(-0.2,3.2)--cycle;
    
    \draw [red,ultra thick] (0,1)--(0,3)--(1,3);
    \draw [red,ultra thick] (1,1)--(2,1)--(2,3);
    
    \foreach \x in {0,...,2}{
    \foreach \y in {0,...,3}{
    \node[bigB] at (\x,\y) {};
    }
    }
    
    \node[wR] at (0,1) {};
    \node[wR] at (1,1) {};
    \node[wR] at (2,1) {};
    
    \node[wR] at (0,2) {};
    \node[wR] at (2,2) {};
    
    \node[wR] at (0,3) {};
    \node[wR] at (1,3) {};
    \node[wR] at (2,3) {};
    
    
    
    
    \end{scope}
    
    \node at (3.7,2.2) {$\phi \circ \rho^{-1} \circ \phi^{-1}$};
    \draw [-stealth, ultra thick] (2.7,1.768)--(4.7,1.768);
    
    \begin{scope}[rotate=45, shift = {(5.5,-5.5)}]
    \draw (0,0) grid (2,3);
    
    \draw [blue, fill = blue, opacity = 0.17] (-0.2,-0.2)--(2.2,-0.2)--(2.2,2.2)--(-0.2,2.2)--cycle;
    
    \draw [red,ultra thick] (0,0)--(0,1)--(1,1)--(1,2);
    \draw [red,ultra thick] (1,0)--(2,0)--(2,2);
    
    \foreach \x in {0,...,2}{
    \foreach \y in {0,...,3}{
    \node[bigB] at (\x,\y) {};
    }
    }
    
    \node[wR] at (0,0) {};
    \node[wR] at (1,0) {};
    \node[wR] at (2,0) {};
    
    \node[wR] at (0,1) {};
    \node[wR] at (1,1) {};
    \node[wR] at (2,1) {};
    
    \node[wR] at (1,2) {};
    \node[wR] at (2,2) {};
    
    
    
    
    \end{scope}
    \end{tikzpicture}
    \caption{The maximum weight of pairs of nonintersecting chains in the shaded rectangle on the left shifts to become the maximum weight of chains in the shaded rectangle on the right. The chains where these maxima are achieved need not correspond.}
    \label{fig:introChainShifting}
\end{figure}
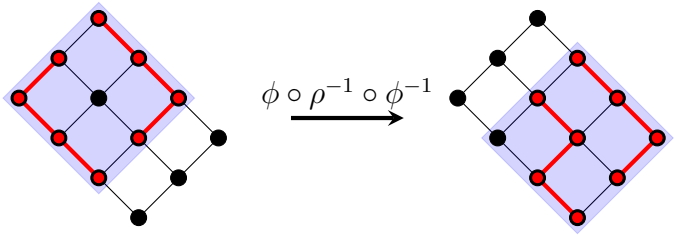

In this paper, we investigate the relationship between rowmotion and another combinatorial map with a birational analogue, namely the \emph{Robinson-Schensted-Knuth correspondence} (RSK). Piecewise-linear and birational liftings of $\RSK$ have been studied previously in, for instance, \cite{danilovkoshevoy,hopkins,noumiyamada}. These maps, along with related maps such as \emph{Sch\"utzenberger promotion} ($\Pro$) acting on semistandard Young tableaux, can be described in terms of toggles as well \cite{johnsonliu}. In Section~\ref{section:chainshifting} we use the toggle perspective to demonstrate how the chain shifting lemma for rowmotion from \cite{johnsonliu} is closely related to a similar chain shifting property of Sch\"utzenberger promotion. 

As an application of these results, we study fillings of \emph{moon polyominoes}, which are diagrams with convex rows and columns that are equivalent to partition diagrams under permuting rows and columns. Rubey \cite{rubey} defines a bijection between nonnegative fillings of equivalent moon polyominoes using Sch\"utzenberger promotion and $\RSK$. This bijection preserves the maximum weight of nonintersecting northeast chains contained in  maximal rectangles of the moon polyominoes.

In Section~\ref{section:fillings} we lift Rubey's results to the piecewise-linear and birational realms, giving new toggle-based proofs. This allows us to construct piecewise-linear bijections between certain analogues of chain polytopes for moon polyominoes that are equivalent under permuting rows and columns. As a result, we deduce that these polytopes have the same Ehrhart polynomial. (Since these polytopes are not in general lattice polytopes, this implies that they exhibit \emph{Ehrhart quasi-polynomial period collapse}.) Rubey also proves that applications of the combinatorial bijection between labelings of several equivalent moon polyominoes commute. In Section~\ref{section:commutationProperties} we use properties of evacuation to give a piecewise-linear and birational proof of this result as well as an analogous result involving rowmotion on the maximal rectangles of the moon polyominoes.

Our proofs rely only on properties of RSK and rowmotion proved in \cite{johnsonliu}, namely Greene's Theorem and the chain shifting lemma, as well as commutation of toggles, but the exact formula for toggles is not otherwise used. As such, our results and proofs all hold for both the piecewise-linear and birational versions of these maps. However, we exclusively phrase our results on the piecewise-linear level to avoid confusion and to more easily discuss polyhedral implications.

The structure of this paper is as follows.
In Section~\ref{section:promotionBackground} we cover background on piecewise-linear rowmotion, promotion, $\RSK$, chain shifting, and fillings of moon polyominoes. In Section~\ref{section:chainshifting} we prove a chain shifting lemma for Sch\"utzenberger promotion. In Section~\ref{section:fillings} we use the chain shifting property to define a piecewise-linear, volume preserving, and continuous lifting of Rubey's map, which restricts to map between certain rational polytopes associated to moon polyominoes. In Section~\ref{section:commutationProperties} we prove chain shifting lemmas for evacuation and Striker-Williams promotion and prove commutation properties for maps on fillings of moon polyominoes.

\section{Background}
\label{section:promotionBackground}

In this section we cover background on the rectangle poset, rowmotion, $\RSK$, promotion, and moon polyominoes.

\subsection{The Rectangle Poset}
\label{subsection:posets}

Let $[r] = \{1,2,\dots,r\}$ be the chain with $r$ elements. Our main poset of study is the \emph{rectangle poset} $R = [r] \times [s]$, the product of two chains. We define the following subsets of interest.

\begin{itemize}
    \item For fixed $i$, the $i$th \emph{up-diagonal} of $R$ is the set of all elements in $R$ of the form $(i,j)$. In Figure~\ref{fig:rectangleposet} the elements of an up-diagonal go southwest to northeast.
    \item For fixed $j$, the $j$th \emph{down-diagonal} of $R$ is the set of all elements in $R$ of the form $(i,j)$. In Figure~\ref{fig:rectangleposet} the elements of a down-diagonal go northwest to southeast.
    \item The $k$th \emph{rank} of $R$ is the set of all elements $(i,j) \in R$ such that $i+j = k$. In Figure~\ref{fig:rectangleposet}, the elements of a rank are horizontally aligned. (Note that the minimum element has rank $2$.)
    \item The $k$th \emph{file} of $R$ is the set of all elements $(i,j) \in R$ such that $i-j = k$. In Figure~\ref{fig:rectangleposet} the elements of a file are vertically aligned.
\end{itemize}

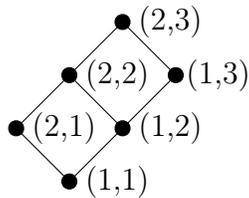
\begin{figure}
    \captionsetup{justification=centering}
    \begin{tikzpicture}[rotate = 45]
    \draw (0,0)--(2,0);
    \draw (0,1)--(2,1);
    
    \draw (0,0)--(0,1);
    \draw (1,0)--(1,1);
    \draw (2,0)--(2,1);
    
    \draw [fill=black] (0,0) circle [radius=0.1cm];
    \draw [fill=black] (1,0) circle [radius=0.1cm];
    \draw [fill=black] (2,0) circle [radius=0.1cm];
    
    \draw [fill=black] (0,1) circle [radius=0.1cm];
    \draw [fill=black] (1,1) circle [radius=0.1cm];
    \draw [fill=black] (2,1) circle [radius=0.1cm];
    
    \node at (0.45,-0.45) {(1,1)};
    \node at (1.45,-0.45) {(1,2)};
    \node at (2.4,-0.4) {(1,3)};
    
    \node at (0.45,0.55) {(2,1)};
    \node at (1.45,0.55) {(2,2)};
    \node at (2.45,0.55) {(2,3)};
    \end{tikzpicture}
    \caption{The rectangle $[2] \times [3]$.}
    \label{fig:rectangleposet}
\end{figure}

We frequently make use of locations of elements in $R$ in relation to other elements, up-diagonals, down-diagonals, and files. 
We say an element $(i_1,j_1)$ is \emph{weakly left} of $(i_2,j_2)$ if $i_1 - j_1 \geq i_2 - j_2$. We say $(i_1,j_1)$ is \emph{strictly left} of $(i_2,j_2)$ if $i_1 - j_1 > i_2 - j_2$. We similarly define \emph{weakly right} and \emph{strictly right}.

\subsection{Rowmotion and Liftings of Rowmotion}
In this subsection, we use $R$ to denote the rectangle poset and $P$ to denote a general poset.

Cameron and Fon-Der-Flaass \cite{cameronfonderflaass} define the following local involution on order ideals.

\begin{Def}
Let $P$ be a finite poset, let $J(P)$ be the set of order ideals of $P$, and let $p \in P$. The \emph{combinatorial toggle} at $p$ is the permutation of order ideals
\[t_p(I) = \begin{cases} I \cup \{p\} &\text{if $p \not\in I$ and $I \cup \{p\} \in J(P)$}, \\
I \setminus \{p\} &\text{if $p \in I$ and $I \setminus \{p\} \in J(P)$}, \\
I &\text{otherwise}.
\end{cases}\]
\end{Def}

Let $n$ be the number of elements of $P$. A \emph{linear extension} of $P$ is an order-preserving bijection $L\colon P \to [n]$.

\begin{Def}
\label{Def:combinatorialRowmotion}
Let $L$ be a linear extension of $P$. Then \emph{combinatorial rowmotion} is
\[\rho = t_{L^{-1}(1)} \circ t_{L^{-1}(2)} \circ \cdots \circ t_{L^{-1}(n)}.\]
\end{Def}
In words, rowmotion is the composition of toggles at all elements in $P$ in the order of a linear extension from the top of $P$ to the bottom. Toggles $t_p$ and $t_q$ commute with each other if and only if $p$ and $q$ do not form a cover relation in $P$. Consequently this definition is independent of the linear extension $L$. Combinatorial rowmotion can also be described without toggles: rowmotion maps an order ideal $I$ to the order ideal generated by the minimal elements of $P \setminus I$; see \cite{cameronfonderflaass}.

Rowmotion has been studied by many authors in connection to dynamical algebraic combinatorics. On most posets, rowmotion is ill-behaved. However, on certain graded posets, rowmotion exhibits a particularly small order. For example, rowmotion on $[r] \times [s]$ has order $r+s$; see \cite{fonderflaass}.

We can encode order ideals as lattice points in $\mathbb{R}^P$. We map an order ideal $I$ to the indicator vector of the order filter $P \setminus I$. Toggles and rowmotion permute the vertices of the following polytope, first defined in \cite{stanley2}.

\begin{Def}
The \emph{order polytope} $\op \subseteq \mathbb{R}^{P}$ of $P$ is the convex hull of the indicator vectors of the order filters of $P$. The inequalities defining $\op$ are $0 \leq x_p \leq 1$ for $p \in P$ and $x_p \leq x_q$ when $p \preceq q$.
\end{Def}

Toggles lift to piecewise-linear, volume-preserving, and continuous maps on $\mathbb{R}^{P}$ \cite{einsteinpropp1}.

\begin{Def}
Let $p \in P$ and let $x \in \mathbb{R}^{P}$. The \emph{piecewise-linear toggle} at $p$ is the map $t_p$ that changes the label at $p$ by
\[x_p \mapsto \min\limits_{q \gtrdot p} x_q + \max\limits_{q \lessdot p} x_q - x_p\]
and changes no other labels. We interpret the empty $\max$ as $0$ and the empty $\min$ as 1.
\end{Def}

(We abuse notation and use the same symbol for combinatorial toggles and piecewise-linear toggles.)
Note that $t_p$ only depends on the coordinates in the neighborhood of $p \in P$ in the Hasse diagram. Consequently for $p,q \in P$, $t_p \circ t_q = t_q \circ t_p$ if and only if neither $p \lessdot q$ nor $p \gtrdot q$ as on the combinatorial level. We frequently reference the following two cases.

\begin{Obs}
\label{Obs:togglesCommute}
Let $A$ and $B$ be compositions of toggles in $[r] \times [s]$, and fix $k \in \mathbb{Z}$.
\begin{enumerate}[(a)]
    \item If every toggle $t_{ij}$ in $A$ satisfies $i > k$ and every toggle $t_{ij}$ in $B$ satisfies $i < k$, then $A \circ B = B \circ A$.
    \item If every toggle $t_{ij}$ in $A$ satisfies $i - j > k$ and every toggle $t_{ij}$ in $B$ satisfies $i - j< k$, then $A \circ B = B \circ A$.
\end{enumerate}
\end{Obs}

In other words, if the toggles in $A$ and the toggles in $B$ are separated by an up-diagonal or a file, then $A$ and $B$ commute with each other. Similar statements hold for down-diagonals and ranks. Just as toggles lift to the piecewise-linear realm, rowmotion also lifts in the same way as in Definition~\ref{Def:combinatorialRowmotion}.

Stanley \cite{stanley2} defines another polytope associated to a poset called the \emph{chain polytope}.

\begin{Def}
The \emph{chain polytope} $\cp \subseteq \mathbb{R}^{P}$ of $P$ is the convex hull of the indicator vectors of the antichains of $P$. Alternatively, $\cp$ is defined by the inequalities $x_p \geq 0$ for all $p \in P$ and $\sum\limits_{p \in C} x_p \leq 1$ for all (maximal) chains $C$ of $P$.
\end{Def}

Stanley defines a piecewise-linear, volume-preserving, and continuous map $\phi\colon \op \to \cp$ called the \emph{transfer map}.

\begin{Def}
The \emph{transfer map} $\phi\colon \op \to \cp$ is given by
\[x_p \mapsto x_p - \max\limits_{q \lessdot p} x_q.\]
\end{Def}
The inverse of the transfer map is
\[x_p \mapsto \max\limits_{q_1 \lessdot q_2 \lessdot \dots \lessdot q_k = p} \sum\limits_{i=1}^k x_{q_i}.\]
On the vertices of $\op$ and $\cp$, the transfer map gives a bijection between the order filters and the antichains of $P$. The transfer map equivariantly induces an \emph{antichain rowmotion} on $P$. See \cite{brouwerschrijver,propproby} for information on combinatorial antichain rowmotion and \cite{josephroby2} for the piecewise-linear and birational liftings. 

In \cite{johnsonliu}, the present authors prove a chain shifting action on $\mathcal{C}(R)$ for $R = [r] \times [s]$. For $u_1 \leq u_2$, $v_1 \leq v_2$, and $k \leq \min(u_2-u_1+1, v_2-v_1+1)$, let $\mathcal{P}_{u_1,v_1}^{u_2,v_2}(k)$ denote the set of all systems of $k$ nonintersecting lattice paths from $(u_1,v_1),(u_1,v_1+1),\dots,(u_1,v_1+k-1)$ to $(u_2,v_2-k+1),(u_2,v_2-k+2),\dots,(u_2,v_2)$. Given a labeling $x \in \mathbb{R}^R$ and some system of paths $\mathcal{L} \in \mathcal{P}_{u_1,v_1}^{u_2,v_2}(k)$, the \emph{weight} of $\mathcal{L}$ is the sum of the weights $x_{ij}$ of all vertices $(i,j)$ in the paths of $\mathcal{L}$. We let $H_{u_1,v_1}^{u_2,v_2}(x;k)$ denote the maximum weight over all systems of paths in $\mathcal{P}_{u_1,v_1}^{u_2,v_2}(k)$. Here the $H$ in the notation is meant to suggest the heaviest or maximum weight collection of chains. (For this reason, we take $H_{u_1,v_1}^{u_2,v_2}(x;k)$ to stabilize for $k \geq \min(u_2-u_1+1,v_2-v_1+1)$ by convention.)

Let $I$ be a size $n$ induced subposet of $R$ and let $L\colon I \to [n]$ be a linear extension. We define the \emph{piecewise-linear rowmotion associated to $I$}, denoted by $\rho_I$, to be
\[\rho_I = t_{L^{-1}(1)} \circ \cdots \circ t_{L^{-1}(n)}\]
where $\rho_I$ acts on $\mathbb{R}^R$. In other words, we toggle from the top of $I$ to the bottom in the order of a linear extension. When $I$ is a principal order ideal, we frequently write only the maximum element in the subscript. For example, the rowmotion associated to the order ideal generated by $(i,j)$ is $\rho_{ij}$.

We have the following chain shifting lemma. (The case $I=R$ was proved by the present authors in \cite[Lemma 4.1]{johnsonliu}.)

\begin{Lemma} [Piecewise-Linear Chain Shifting] 
\label{Lemma:orderidealshifting}
Suppose $1 < u \leq v \leq r$. If $I \subseteq R$ satisfies $[v-1] \times [s] \subseteq I$, then for any $k \in \mathbb{Z}_{\geq 0}$ and any $x \in \mathbb{R}^R$,
\[H_{u,1}^{v,s}(x;k) = H_{u-1,1}^{v-1,s}(\phi \circ \rho_I^{-1} \circ \phi^{-1}(x);k).\]
\end{Lemma}
On the left hand side we have the maximum weight of all sets of $k$ nonintersecting lattice paths in $\mathcal{P}_{u,1}^{v,s}(k)$. On the right hand side we see that after applying $\phi \circ \rho^{-1} \circ \phi^{-1}$ to $x$, the endpoints of the lattice paths of the maximum weight sum have shifted down by one in the rectangle.

\begin{Ex}
We can visualize this in the following example. Consider the labeling of $R = [4] \times [3]$ appearing in Figure~\ref{fig:chainShifting} and its image under $\phi \circ \rho^{-1} \circ \phi^{-1}$ (here $I = R$) with $u=2$ and $v=4$. Adding the weights along the red chains in both labelings gives 1.15. Note that this is the maximum possible sum of weights among pairs of chains in both rectangles. Visually, the maximum weight of two nonintersecting chains appearing in an interval shifts down in the poset. The chains achieving the maximum weight in the shifted subrectangle need not be shifts of the maximum weight chains in the original subrectangle.
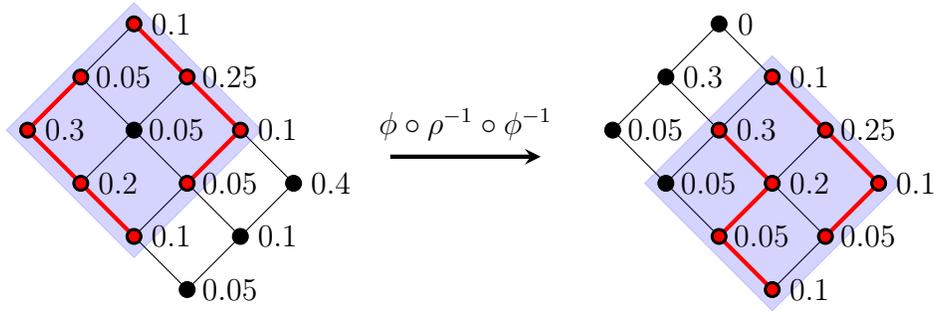
\begin{figure}
    \centering
    \begin{tikzpicture}
    \begin{scope}[rotate=45]
    \draw (0,0) grid (2,3);
    
    \draw [blue, fill = blue, opacity = 0.17] (-0.2,0.8)--(2.2,0.8)--(2.2,3.2)--(-0.2,3.2)--cycle;
    
    \draw [red,ultra thick] (0,1)--(0,3)--(1,3);
    \draw [red,ultra thick] (1,1)--(2,1)--(2,3);
    
    \foreach \x in {0,...,2}{
    \foreach \y in {0,...,3}{
    \node[bigB] at (\x,\y) {};
    }
    }
    
    \node[wR] at (0,1) {};
    \node[wR] at (1,1) {};
    \node[wR] at (2,1) {};
    
    \node[wR] at (0,2) {};
    \node[wR] at (2,2) {};
    
    \node[wR] at (0,3) {};
    \node[wR] at (1,3) {};
    \node[wR] at (2,3) {};
    
    \node at (0.4,-0.4) {0.05};
    \node at (1.35,-0.35) {0.1};
    \node at (2.35,-0.35) {0.4};
    
    \node at (0.35,0.65) {0.1};
    \node at (1.4,0.6) {0.05};
    \node at (2.35,0.65) {0.1};
    
    \node at (0.35,1.65) {0.2};
    \node at (1.4,1.6) {0.05};
    \node at (2.4,1.6) {0.25};
    
    \node at (0.35,2.65) {0.3};
    \node at (1.4,2.6) {0.05};
    \node at (2.35,2.65) {0.1};
    \end{scope}
    
    \node at (3.7,2.2) {$\phi \circ \rho^{-1} \circ \phi^{-1}$};
    \draw [-stealth, ultra thick] (2.7,1.768)--(4.7,1.768);
    
    \begin{scope}[rotate=45, shift = {(5.5,-5.5)}]
    \draw (0,0) grid (2,3);
    
    \draw [blue, fill = blue, opacity = 0.17] (-0.2,-0.2)--(2.2,-0.2)--(2.2,2.2)--(-0.2,2.2)--cycle;
    
    \draw [red,ultra thick] (0,0)--(0,1)--(1,1)--(1,2);
    \draw [red,ultra thick] (1,0)--(2,0)--(2,2);
    
    \foreach \x in {0,...,2}{
    \foreach \y in {0,...,3}{
    \node[bigB] at (\x,\y) {};
    }
    }
    
    \node[wR] at (0,0) {};
    \node[wR] at (1,0) {};
    \node[wR] at (2,0) {};
    
    \node[wR] at (0,1) {};
    \node[wR] at (1,1) {};
    \node[wR] at (2,1) {};
    
    \node[wR] at (1,2) {};
    \node[wR] at (2,2) {};
    
    \node at (0.35,-0.35) {0.1};
    \node at (1.4,-0.4) {0.05};
    \node at (2.35,-0.35) {0.1};
    
    \node at (0.4,0.6) {0.05};
    \node at (1.35,0.65) {0.2};
    \node at (2.4,0.6) {0.25};
    
    \node at (0.4,1.6) {0.05};
    \node at (1.35,1.65) {0.3};
    \node at (2.35,1.65) {0.1};
    
    \node at (0.4,2.6) {0.05};
    \node at (1.35,2.65) {0.3};
    \node at (2.25,2.75) {0};
    \end{scope}
    \end{tikzpicture}
    \caption{The maximum weight pair of nonintersecting paths in each blue rectangle is highlighted in red. By Lemma~\ref{Lemma:orderidealshifting}, these weights are the same.}
    \label{fig:chainShifting}
\end{figure}
\end{Ex}

\begin{proof}[Proof of Lemma~\ref{Lemma:orderidealshifting}]
In the case of $I=R$, it follows from \cite[Lemma~4.1]{johnsonliu} that
\[H_{u,1}^{v,s}(x;k) = H_{u-1,1}^{v-1,s}(\phi \circ \rho^{-1} \circ \phi^{-1}(x);k).\]
The weight on the right hand side is independent of the labels of coordinates outside of the order ideal $[v-1] \times [s]$, so replacing $\rho^{-1}$ in the right hand side with $\rho^{-1}_I$ for any $I \supseteq [v-1] \times [s]$ does not change the right hand side.
\end{proof}

In this paper, we exclusively prove results in the piecewise-linear setting for consistency. However, most of our results (except for those specifically referencing a polytope in Section~\ref{section:fillings}) lift to the birational level as well. We briefly review birational rowmotion.

Tropical expressions and functions in terms of $\max$, $\min$, $+$, and $-$ can be \emph{detropicalized} by replacing $\max$ with addition, addition with multiplication, and subtraction with division. Note that for all $a,b \in \mathbb{R}$,
\[\min(a,b) = -\max(-a,-b).\]
Then the minimum detropicalizes to the \emph{parallel sum} $\left(a^{-1} + b^{-1}\right)^{-1}$. Piecewise-linear toggles detropicalize to the following birational toggles \cite{einsteinpropp1}.

\begin{Def}
Let $P$ be a finite poset, let $p \in P$, and let $x \in \mathbb{R}^{P}_{>0}$. The \emph{birational toggle} at $p$ is the map $t_p$ that changes the $p$-coordinate of $x$ by
\[x_p \mapsto \left(\frac{1}{\sum\limits_{q \gtrdot p} \frac{1}{x_q}} \right) \left( \sum\limits_{q \lessdot p} x_q \right) \frac{1}{x_p}\]
and changes no other coordinates. (We interpret the empty sum as 1.)
\end{Def}

We can lift rowmotion as well by the definition analogous to Definition~\ref{Def:combinatorialRowmotion}. Typically, it is easier to work on the birational level than it is to work on the piecewise-linear level. Since any subtraction-free rational formula tropicalizes, results on the piecewise-linear level are frequently inherited from the birational level. In \cite{johnsonliu}, we proved the birational versions of Lemma~\ref{Lemma:orderidealshifting} and Greene's theorem (to be described below).
Most of the results in this paper will rely primarily on these two results and hence will hold in both the birational and the piecewise-linear settings.

\subsection{RSK and Promotion}
\label{subsection:rskandpromotion}
In this subsection we discuss piecewise-linear RSK in terms of toggles as described by \cite{hopkins, johnsonliu, pak}. We also describe piecewise-linear Greene's theorem and define piecewise-linear promotion. See \cite{sagan} for background on standard Young tableaux and classical RSK.

Classical $\RSK$ maps a matrix $A \in \mathbb{Z}_{\geq 0}^{n \times n}$ to a pair of semistandard Young tableaux $(P,Q)$ of the same shape, called the \emph{insertion tableau} and the \emph{recording tableau} respectively. We can represent a semistandard Young tableaux with a \emph{Gelfand-Tsetlin pattern}.

\begin{Def}
A \emph{Gelfand-Tsetlin pattern} is a triangular array of nonnegative integers

\begin{center}
\begin{tikzpicture}
\begin{scope}[xscale=-1]
\node at (0,0) {$g_{1,1}$};
\node at (1.8,0) {$g_{2,2}$};
\node at (3.6,0) {$g_{3,3}$};
\node at (5.4,0) {$\cdots$};
\node at (7.2,0) {$g_{n,n}$};

\node at (0.9,-0.5) {$g_{2,1}$};
\node at (2.7,-0.5) {$g_{3,2}$};
\node at (4.5,-0.5) {$\cdots$};
\node at (6.3,-0.5) {$g_{n,n-1}$};

\node at (1.8,-1) {$g_{3,1}$};
\node at (3.6,-1) {$\cdots$};
\node at (5.4,-1) {$g_{n,n-2}$};

\node at (2.7,-1.5) {$\iddots$};
\node at (4.5,-1.5) {$\ddots$};

\node at (3.6,-2) {$g_{n,1}$};
\end{scope}
\end{tikzpicture}
\end{center}
\noindent such that $g_{i,j} \geq g_{i,j-1}$ and $g_{i,j} \geq g_{i-1,j}$ for all $i,j$.
\end{Def}

If $P$ is a semistandard Young tableau, then 
the Gelfand-Tsetlin pattern associated to $P$ has row $i$ given by the shape of the boxes of $P$ with label at most $n-i+1$. Since $\RSK$ maps a matrix to two semistandard Young tableaux of the same shape, the Gelfand-Tsetlin patterns of these tableaux have the same first row. We form a matrix $\RSK(A)$ by gluing the Gelfand-Tsetlin patterns for $P$ and $Q$ together along their common first row. By convention, we glue so that the Gelfand-Tsetlin pattern of the $P$-tableau is a labeling of the set of points $(i,j) \in R$ that lie weakly left of $(n,n)$, and similarly the Gelfand-Tsetlin pattern of the $Q$-tableau is a labeling of the points lying weakly right of $(n,n)$. The inequalities in the definition of the Gelfand-Tsetlin pattern imply that the rows and columns of $\RSK(A)$ are weakly increasing.

Labelings of the rectangle poset $[r] \times [s]$ can be realized as $r \times s$ matrices. Hopkins \cite{hopkins} shows that this version of $\RSK$ has a description in terms of toggles in the rectangle poset. In \cite{johnsonliu} we show that this description can be rearranged so that the transfer map factors out, giving the following expression which we take to be the definition of RSK.

\begin{Def}
Let $r,s \in \mathbb{Z}_{> 0}$, $m=\min\{r,s\}$, and $R = [r] \times [s]$. Then $\RSK$ is defined to be the following composition:
\[\RSK = \rho_{r-m+1,s-m+1}^{-1} \circ \cdots \circ \rho_{r-2,s-2}^{-1} \circ \rho_{r-1,s-1}^{-1} \circ \phi^{-1}.\]
\end{Def}
For a labeling $x$ of $R$, we define the following two coordinate projections of $\RSK(x)$:
\begin{align*}
    P(x) &= (\RSK(x)_{ij}:i-j \geq r-s),\\
    Q(x) &= (\RSK(x)_{ij}:i-j \leq r-s).
\end{align*}
On the combinatorial level, $P(x)$ and $Q(x)$ are the Gelfand-Tsetlin patterns of $P$- and $Q$-tableaux. As an abuse of language, we refer to $P(x)$ and $Q(x)$ as the $P$- and $Q$-tableaux of $x$ on the piecewise-linear level.

Classical results about $\RSK$ also lift to the piecewise-linear and birational realms. For example, we have the following version of Greene's theorem \cite{johnsonliu, noumiyamada}.

\begin{Th}[Piecewise-Linear Greene's Theorem] \label{Th:plGreene}
Let $R = [r] \times [s]$, and choose $(i,j) \in R$ such that either $i = r$ or $j = s$. Then for any $x \in \mathbb R^{R}$ and $1\leq k \leq \min \{i,j\}$,
\[\sum_{t=0}^{k-1} \RSK(x)_{i-t,j-t} = H_{1,1}^{i,j}(x;k).\]
\end{Th}

Very similar versions of this in the combinatorial realm appear in \cite{krattenthaler,rubey}.

\begin{Ex}
Consider the leftmost labeling $x$ in Figure~\ref{fig:RSKandPromotion}. One can compute $\RSK(x)$ by applying $\phi^{-1}$ followed by a sequence of toggles. Alternatively, we can compute coordinates of $\RSK(x)$ by Greene's theorem. The maximum weight among all paths from $(1,1)$ to $(3,3)$ in the leftmost labeling is 5, and so $\RSK(x)_{3,3} = 5$. The maximum weight of all pairs of nonintersecting paths from $\{(1,1),(1,2)\}$ to $\{(3,2),(3,3)\}$ is $6$. Hence
\[\RSK(x)_{2,2} = H_{1,1}^{3,3}(x;2) - H_{1,1}^{3,3}(x;1) = 6 - 5 = 1.\]
All coordinates of $\RSK(x)$ can be computed via this method.

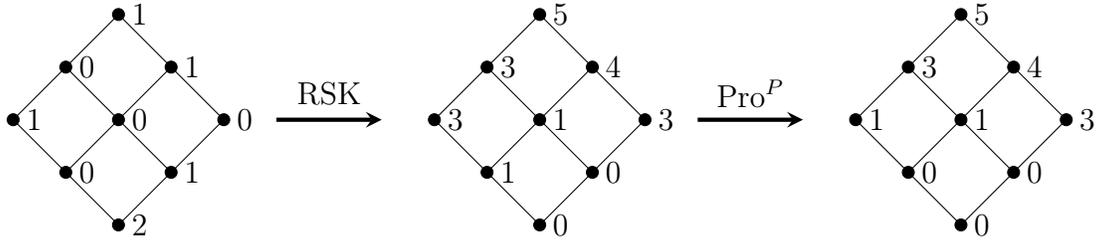
\begin{figure}
\centering
\begin{tikzpicture}[scale = 0.7]

\begin{scope}[shift={(0,0)}]
\draw (0,0)--(2,2);
\draw (-1,1)--(1,3);
\draw (-2,2)--(0,4);

\draw (0,0)--(-2,2);
\draw (1,1)--(-1,3);
\draw (2,2)--(0,4);

\node[wB] at (0,0) {};
\node[wB] at (1,1) {};
\node[wB] at (2,2) {};

\node[wB] at (-1,1) {};
\node[wB] at (0,2) {};
\node[wB] at (1,3) {};

\node[wB] at (-2,2) {};
\node[wB] at (-1,3) {};
\node[wB] at (0,4) {};

\node at (0.4,0) {2};
\node at (1.4,1) {1};
\node at (2.4,2) {0};

\node at (-0.6,1) {0};
\node at (0.4,2) {0};
\node at (1.4,3) {1};

\node at (-1.6,2) {1};
\node at (-0.6,3) {0};
\node at (0.4,4) {1};
\end{scope}

\draw [ultra thick, -stealth] (3,2)--(5,2);
\node at (4,2.5) {$\RSK$};

\begin{scope}[shift={(8,0)}]
\draw (0,0)--(2,2);
\draw (-1,1)--(1,3);
\draw (-2,2)--(0,4);

\draw (0,0)--(-2,2);
\draw (1,1)--(-1,3);
\draw (2,2)--(0,4);

\node[wB] at (0,0) {};
\node[wB] at (1,1) {};
\node[wB] at (2,2) {};

\node[wB] at (-1,1) {};
\node[wB] at (0,2) {};
\node[wB] at (1,3) {};

\node[wB] at (-2,2) {};
\node[wB] at (-1,3) {};
\node[wB] at (0,4) {};

\node at (0.4,0) {0};
\node at (1.4,1) {0};
\node at (2.4,2) {3};

\node at (-0.6,1) {1};
\node at (0.4,2) {1};
\node at (1.4,3) {4};

\node at (-1.6,2) {3};
\node at (-0.6,3) {3};
\node at (0.4,4) {5};
\end{scope}

\draw [ultra thick, -stealth] (11,2)--(13,2);
\node at (12,2.5) {$\proid$};

\begin{scope}[shift={(16,0)}]
\draw (0,0)--(2,2);
\draw (-1,1)--(1,3);
\draw (-2,2)--(0,4);

\draw (0,0)--(-2,2);
\draw (1,1)--(-1,3);
\draw (2,2)--(0,4);

\node[wB] at (0,0) {};
\node[wB] at (1,1) {};
\node[wB] at (2,2) {};

\node[wB] at (-1,1) {};
\node[wB] at (0,2) {};
\node[wB] at (1,3) {};

\node[wB] at (-2,2) {};
\node[wB] at (-1,3) {};
\node[wB] at (0,4) {};

\node at (0.4,0) {0};
\node at (1.4,1) {0};
\node at (2.4,2) {3};

\node at (-0.6,1) {0};
\node at (0.4,2) {1};
\node at (1.4,3) {4};

\node at (-1.6,2) {1};
\node at (-0.6,3) {3};
\node at (0.4,4) {5};
\end{scope}

\end{tikzpicture}
\caption{An example of piecewise-linear $\RSK$ and $\proid$ applied to a labeling.}
\label{fig:RSKandPromotion}
\end{figure}
\end{Ex}

Another important map on semistandard Young tableaux is Sch\"utzenberger promotion. The piecewise-linear analogue is the following.

\begin{Def}
\label{Def:PLPro}
Let $S$ be the subposet of elements weakly left of $(r,s)$ in $R$, and let $F_k$ be the $k$th file of $S$. Then \emph{piecewise-linear promotion} $\Pro\colon\mathbb{R}^S \to \mathbb{R}^S$ is given by
\[\Pro = \rho_{F_{r-s+1}} \circ \cdots \circ \rho_{F_{r-1}}.\]
We write $\proid \colon \mathbb R^R \to \mathbb R^R$ for the map that applies $\Pro$ to $S$ and keeps the labels of $R \setminus S$ fixed. (We define $\idpro$ analogously by transposition.)
\end{Def}

This definition does not depend on the order of the toggles in any given $F_k$ since elements in a file do not form cover relations with one another. If a labeling of $S$ is a Gelfand-Tsetlin pattern, then $\rho_{F_k}$ is the $k$th Bender-Knuth involution \cite[Proposition 2.4]{kirillov}. See Figure~\ref{fig:RSKandPromotion} for an example of this map. 

Rubey \cite{rubey} shows that on nonnegative integer labelings of $R$, $\RSK^{-1} \circ \proid \circ \RSK$ cyclically shifts the down-diagonal sums of the labeling and preserves the up-diagonal sums. In Section~\ref{section:chainshifting} we give a piecewise-linear proof of this chain shifting result for all labelings of $R$ with nonnegative real entries using toggles.

\subsection{Moon Polyominoes}

A \emph{polyomino} is a finite subset of $\mathbb{Z} \times \mathbb{Z}$. We represent polyominoes with boxes and draw them tilted, so that boxes form up- and down-diagonals analogous to the rectangle poset as in Figure~\ref{fig:polyominoes}. We say a polyomino is \emph{convex} if each up- and down-diagonal of the polyomino is connected. A polyomino $\mathcal{M}$ is \emph{intersection-free} if for any two $i_1,i_2 \in \mathbb{Z}$, either
\[\{j \in \mathbb{Z}:(i_1,j) \in \mathcal{M}\} \subseteq \{j \in \mathbb{Z}:(i_2,j) \in \mathcal{M}\}\]
or the reverse inclusion holds. A \emph{moon polyomino} is a convex and intersection-free polyomino. For example, in Figure~\ref{fig:polyominoes}, the left diagram is a moon polyomino, but the other two are not: in the center diagram, the top up diagonal is not convex, while in the right diagram, the two up-diagonals do not form an inclusion relation (as subsets of $\mathbb Z$). A \emph{filling} of a moon polyomino $\m$ is a labeling $x \in \mathbb{Z}_{\geq 0}^{\m}$ of its boxes, as shown in the left diagram of Figure~\ref{fig:polyominoes}.

Any maximal subrectangle $R$ of a moon polyomino naturally has a rectangle poset structure. Combinatorial statistics related to maximal rectangles are of particular interest in the study of moon polyominoes \cite{rubey}.

\begin{Def}
Let $\mathcal{M}$ be a moon polyomino, $k \in \mathbb{Z}_{\geq 0}$, and $x \in \mathbb{Z}^\mathcal{M}_{\geq 0}$ a filling of $\mathcal{M}$.
\begin{enumerate}[(a)]
    \item A \emph{weak northeast chain} of $x$ of weight $k$ is a chain in some rectangle $R \subseteq \mathcal{M}$ such that the sum of all labels in the chain is $k$.
    \item A \emph{strict southeast chain} of $x$ of size $k$ is an antichain of $k$ elements in some rectangle $R \subseteq \mathcal{M}$ such that each element of the antichain has a nonzero label.
\end{enumerate}
\end{Def}

Analogous to weights in the rectangle poset, we let $H_\m(x;k)$ denote the maximum weight of $k$ disjoint northeast chains contained in a common rectangle in $\m$. We let $L_\m(x)$ denote the maximum size among all (strict) southeast chains in $\m$. Here the $L$ in the notation is meant to suggest the longest or maximum size southeast chain.

\begin{figure}
    \centering
    \begin{tikzpicture}[scale = 0.68, rotate = 135]
    \label{polyominoes}
    \draw [very thick] (0,0)--(3,0);
    \draw [very thick] (0,-1)--(3,-1);
    \draw [very thick] (0,-2)--(3,-2);
    \draw [very thick] (1,-3)--(2,-3);
    
    \draw [very thick] (0,0)--(0,-2);
    \draw [very thick] (1,0)--(1,-3);
    \draw [very thick] (2,0)--(2,-3);
    \draw [very thick] (3,0)--(3,-2);
    
    \node at (0.5,-0.5) {2};
    \node at (1.5,-0.5) {0};
    \node at (2.5,-0.5) {0};
    
    \node at (0.5,-1.5) {0};
    \node at (1.5,-1.5) {1};
    \node at (2.5,-1.5) {0};
    
    \node at (1.5,-2.5) {1};
    
    
    
    
    \end{tikzpicture}\begin{tikzpicture}[scale = 0.68]
    \draw [opacity = 0] (1,-4)--(1,-1);
    
    \draw [opacity = 0] (0,-2)--(2,-2);
    
    \end{tikzpicture}\begin{tikzpicture}[scale = 0.68,rotate = 135]
    
    \draw [very thick] (0,0)--(2,0);
    \draw [very thick] (0,-1)--(2,-1);
    \draw [very thick] (0,-2)--(2,-2);
    \draw [very thick] (1,-3)--(2,-3);
    
    \draw [very thick] (0,0)--(0,-2);
    \draw [very thick] (1,0)--(1,-3);
    \draw [very thick] (2,0)--(2,-1);
    \draw [very thick] (2,-2)--(2,-3);
    
    
    \end{tikzpicture}\begin{tikzpicture}[scale = 0.68]
    \draw [opacity = 0] (1,-4)--(1,-1);
    
    \draw [opacity = 0] (0,-2)--(2,-2);
    
    \end{tikzpicture}\begin{tikzpicture}[scale = 0.68,rotate = 135]
    
    \draw [very thick] (0,0)--(1,0);
    \draw [very thick] (0,-1)--(2,-1);
    \draw [very thick] (0,-2)--(2,-2);
    \draw [very thick] (1,-3)--(2,-3);
    
    \draw [very thick] (0,0)--(0,-2);
    \draw [very thick] (1,0)--(1,-3);
    \draw [very thick] (2,-1)--(2,-3);
    
    \end{tikzpicture}
    \caption{Three polyominoes and a filling. The first is a moon polyomino (with a filling) while the other two are not.}
    \label{fig:polyominoes}
\end{figure}
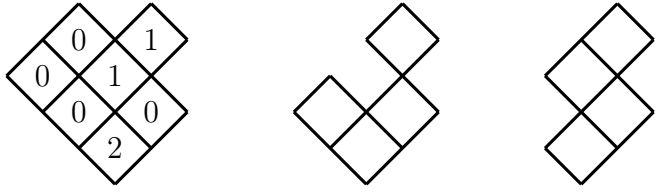

\begin{Ex}
In the filling in Figure~\ref{fig:polyominoes}, $H_\m(x;1) = 3$. Note that the two 1s and the 2 do not all form a single northeast chain since they do not lie in a common rectangle. Also $L_\m(x) = 1$. The 2 does not contribute more than 1 to the longest strict southeast chain since strict chains are not weighted by entries; all nonzero entries count the same.
\end{Ex}

\begin{Def}
We say that two moon polyominoes $\mathcal{M}$ and $\mathcal{N}$ are \emph{equivalent} if $\mathcal{N}$ can be obtained from $\mathcal{M}$ by permuting up-diagonals and down-diagonals.
\end{Def}

Let $\m$ and $\n$ be moon polyominoes, let $R$ be a maximal rectangle of $\m$, and let $\mathcal{D}$ be the union of all down-diagonals of $\m$ intersecting $R$. Suppose further that we can obtain $\n$ from $\m$ as the image of the following map on $\m$:
\[(i,j) \mapsto \begin{cases} (i,j-1)&\text{if $(i,j) \in \mathcal{D} \setminus R$,} \\
(i,j)&\text{otherwise}.\end{cases}\]
In words, we shift the boxes in the down-diagonals containing $R$ that lie outside of $R$ downward. For instance, we could transform $\m$ into $\n$ as follows:
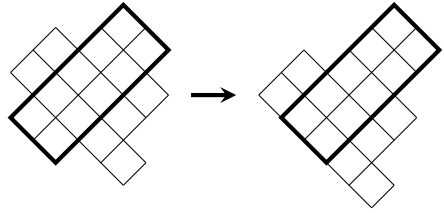
\begin{center}
\begin{tikzpicture}[scale = 0.3]
\begin{scope}[shift={(0,0)}]
\draw (1,-1)--(2,0);
\draw (0,0)--(3,3);
\draw (-2,0)--(3,5);
\draw (-3,1)--(2,6);
\draw (-4,2)--(1,7);
\draw (-4,4)--(-2,6);

\draw (-2,0)--(-4,2);
\draw (1,-1)--(-4,4);
\draw (2,0)--(-3,5);
\draw (2,2)--(-2,6);
\draw (3,3)--(0,6);
\draw (3,5)--(1,7);

\draw [ultra thick](-2,0)--(3,5)--(1,7)--(-4,2)--cycle;
\end{scope}

\draw [ultra thick, -stealth] (4,3)--(6,3);

\begin{scope}[shift={(12,0)}]
\draw (0,-2)--(1,-1);
\draw (-1,-1)--(2,2);
\draw (-2,0)--(3,5);
\draw (-3,1)--(2,6);
\draw (-4,2)--(1,7);
\draw (-5,3)--(-3,5);


\draw (0,-2)--(-5,3);
\draw (1,-1)--(-4,4);
\draw (1,1)--(-3,5);
\draw (2,2)--(-1,5);
\draw (2,4)--(0,6);
\draw (3,5)--(1,7);

\draw [ultra thick](-2,0)--(3,5)--(1,7)--(-4,2)--cycle;
\end{scope}
\end{tikzpicture}
\end{center}
Then $\m$ and $\n$ are equivalent. Any equivalence of moon polyominoes can be described in terms of shifts of down-diagonals or up-diagonals of this form.

In \cite{rubey}, Rubey defines a bijection from fillings of $\m$ to fillings of $\n$ that preserves $H_\m(x;k)$ and $L_\m(x)$, which we now describe. Let $\Omega = \RSK^{-1} \circ \proid \circ \RSK$ on fillings of $R$ (Rubey denotes this map by $\pi$). Let $\m$, $\n$, and $R$ be as above and let $\mathcal{E}$ denote the down-diagonals of $\n$ intersecting $R$. Extend the action of $\Omega$ on fillings of $R$ to a bijection between fillings of $\m$ and $\n$ as follows:
\[\Omega_{\m \to \n}(x)_{ij} = \begin{cases} \Omega(x)_{ij}&\text{if $(i,j) \in R$,} \\
x_{i,j+1}&\text{if $(i,j) \in \mathcal{E} \setminus R$,} \\
x_{ij}&\text{if $(i,j) \in \n \setminus \mathcal{E}$.} \end{cases}\]
In other words, apply $\Omega$ inside of $R$ and shift the labels outside of $R$ appropriately. (One can symmetrically define an extension of $\RSK^{-1} \circ \idpro \circ \RSK$ on $R$ to shift up-diagonals.)

In Section \ref{section:fillings} we show that $\Omega_{\m \to \n}$ lifts to a piecewise-linear, volume-preserving, continuous map between non-lattice polytopes, and we give piecewise-linear generalizations of Rubey's preserved statistics. Rubey also shows that applications of $\Omega_{\m \to \n}$ to distinct maximal rectangle of a moon polyomino commute with one another. We prove the piecewise-linearization of this in Section~\ref{section:commutationProperties}. Our methods differ significantly from Rubey's. We use very few of the combinatorial properties of $\RSK$; instead we rely near exclusively on commutation of toggles, the chain shifting properties of rowmotion, and piecewise-linear Greene's theorem. Consequently all proofs in Sections~\ref{section:chainshifting} and \ref{section:commutationProperties} are valid in the combinatorial, piecewise-linear, and birational realms.

\section{Promotion and Chain Shifting}
\label{section:chainshifting}
Rubey's chain shifting property can be derived completely on the piecewise-linear (or birational) level in terms of the rowmotion chain shifting property and RSK. All proofs in this section are valid in both the piecewise-linear and birational realms. 

We begin by studying a slight variant of the map $\RSK$. To each $(a,b) \in [r+1] \times [s+1]$, define
\[\RSK_{a,b} = \rho^{-1}_{a-\min(a,b)+1,b-\min(a,b)+1} \circ \dots \circ \rho^{-1}_{a-1,b-1} \circ \phi^{-1},\]
so that $\RSK = \RSK_{r,s}$. 
Note that we allow $a = r+1$ and $b = s+1$, so that $(a,b)$ may lie just outside of the rectangle $R$. Let $\mathcal{RSK}_{a,b}$ denote ${\RSK_{a,b}} \circ \phi$, the toggle part of $\RSK_{a,b}$. Note that
\[\mathcal{RSK}_{r+1,s} = \mathcal{RSK}_{r,s-1} \circ \rho^{-1}_{r,s-1}.\]
See Figure~\ref{fig:RSKNesting}. When $(a,b)$ lies inside of the rectangle, we have the following proposition.

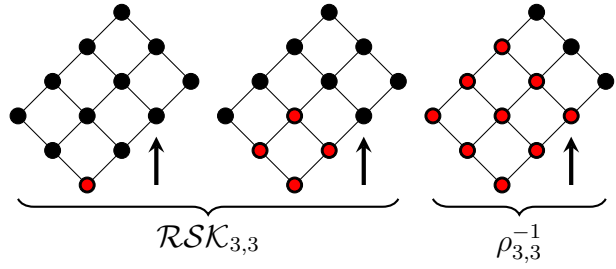
\begin{figure}
    \centering
    \begin{tikzpicture}[scale = 0.65]
    
        \tikzstyle{fancyR}=[circle, draw, fill=red, inner sep=0pt, minimum width=4.6pt]
        \foreach \z in {0,3,6}{
        \begin{scope}[rotate = 45, shift = {(\z,-\z)}]
            \draw (0,0) grid (3,2);
            
            \foreach \x in {0,...,3}{
                \foreach \y in {0,...,2}{
                    \node[bigB] at (\x,\y) {};
                }
            }
        \end{scope}
        }
        
        \begin{scope}[rotate = 45, shift={(0,0)}]
        
        \node[fancyR] at (0,0) {};
        
        \draw [-stealth,shift={(1,-1)},ultra thick] (0,0)--(0.7,0.7);
        
        \draw [decorate,
        decoration = {brace,mirror,amplitude=6pt}, thick] (-1.2,0.8) --  (4.3,-4.7);
        \node at (0.95,-2.55) {$\mathcal{RSK}_{3,3}$};
        \end{scope}
        
        \begin{scope}[rotate = 45, shift={(3,-3)}]
        
        \node[fancyR] at (0,0) {};
        \node[fancyR] at (1,0) {};
        
        \node[fancyR] at (0,1) {};
        \node[fancyR] at (1,1) {};
        
        \draw [-stealth,shift={(1,-1)},ultra thick] (0,0)--(0.7,0.7);
        \end{scope}
        
        \begin{scope}[rotate = 45, shift={(6,-6)}]
        
        \node[fancyR] at (0,0) {};
        \node[fancyR] at (1,0) {};
        \node[fancyR] at (2,0) {};
        
        \node[fancyR] at (0,1) {};
        \node[fancyR] at (1,1) {};
        \node[fancyR] at (2,1) {};
        
        \node[fancyR] at (0,2) {};
        \node[fancyR] at (1,2) {};
        \node[fancyR] at (2,2) {};
        
        \draw [decorate,
        decoration = {brace,mirror,amplitude=6pt}, thick] (-1.2,0.8) --  (1.3,-1.7);
        \node at (-0.6,-1.1) {$\rho^{-1}_{3,3}$};
        
        \draw [-stealth,shift={(1,-1)},ultra thick] (0,0)--(0.7,0.7);
        \end{scope}
        
    \end{tikzpicture}
    \caption{An example of $\mathcal{RSK}_{r+1,s}$ where $r = 3$ and $s = 4$. Toggles are composed in the order indicated by the arrow beneath each rectangle, from the rightmost rectangle to the leftmost.}
    \label{fig:RSKNesting}
\end{figure}

\begin{Prop}
\label{miniRSK}
Let $(a,b) \in R$ and let $\pi\colon\mathbb{R}^R \to \mathbb{R}^{[a] \times [b]}$ be a coordinate projection. Then for all $(i,j) \in R$,
\[\RSK_{a,b}(x)_{ij} = \begin{cases} {\RSK} \circ \pi(x)_{ij} &\text{if  $(i,j) \in [a] \times [b]$}, \\
\phi^{-1}(x)_{ij} &\text{otherwise},\end{cases}\]
where the $\RSK$ on the right hand side acts on the rectangle $[a] \times [b]$.
\end{Prop}

\begin{proof}
Since $\RSK_{a,b}$ contains no toggles in up-diagonals $a$ through $r$ or in down-diagonals $b$ through $s$, we know that for $(i,j) \in R \setminus \left([a] \times [b] \right)$,
\[\RSK_{a,b}(x)_{ij} = \phi^{-1}(x)_{ij}.\]
For $(i,j) \in [a] \times [b]$, note that toggles appearing in $\RSK_{a,b}$ do not depend on elements in $R \setminus \left([a] \times [b]\right)$, so $\RSK_{a,b}(x)_{ij} = \RSK(x)_{ij}$.
\end{proof}


If we compose instances of $\mathcal{RSK}_{a,b}$ and $\mathcal{RSK}_{c,b}^{-1}$, a great deal of cancellation occurs since these maps are nearly inverses of each other. When $a$ and $c$ differ by 1, we can give an explicit description of the toggles that do not cancel.

As in Section~\ref{section:promotionBackground}, let $S$ be the induced subposet of $[r] \times [s]$ of elements weakly left of $(r,s)$, and let $F_k$ be the file indexed by $k$. 
If we now let $F_k^{ij} = F_k \cap ([i] \times [j])$, then define
\[\Pro_{ij} = \rho_{F_{i-j+1}^{ij}} \circ \cdots \circ \rho_{F_{i-2}^{ij}} \circ \rho_{F_{i-1}^{ij}}\]
on $\mathbb R^S$, and similarly define $\proid_{ij}$ on $\mathbb R^R$. Hence $\proid_{r,s} = \proid$. We also define $\idpro_{ij}$ by transposing appropriately.

\begin{Lemma}
\label{Lemma:proidRSK}
Let $(a,b) \in R$. Then
\[\mathcal{RSK}_{a+1,b} \circ \mathcal{RSK}_{a,b}^{-1} = {\RSK_{a+1,b}} \circ \RSK_{a,b}^{-1}= \proid_{a,b}.\]
\end{Lemma}


To visualize Lemma~\ref{Lemma:proidRSK} in the case $(a,b)=(r,s)$, refer to Figure~\ref{fig:RSKNesting} for a depiction of the toggles in $\mathcal {RSK}_{r+1,s}$. Peel off the top up-diagonal of toggles in each rectangle and commute them to the left past the other toggles as much as possible. This yields ${\proid} \circ \mathcal{RSK}$. (See also Figure~\ref{fig:doubleRSKCancelation} for an alternate depiction.)

\begin{proof}
Let $m = \min(a+1,b)$. For $k \leq a$, let $I_k = \{k\} \times [b-a+k-1]$, so that in particular $\proid_{a,b}$ can be written as $\rho^{-1}_{I_{a-m+2}} \circ \dots \circ \rho^{-1}_{I_{a-1}} \circ \rho^{-1}_{I_{a}}$. Then 
\begin{align*}
    \mathcal{RSK}_{a+1,b} &= \rho^{-1}_{a-m+2,b-m+1} \circ  \dots \circ \rho^{-1}_{a-1,b-2} \circ \rho^{-1}_{a,b-1} \\
    &= \left(\rho^{-1}_{I_{a-m+2}} \circ \rho^{-1}_{a-m+1,b-m+1}\right) \circ \dots \circ (\rho^{-1}_{I_{a-1}} \circ \rho^{-1}_{a-2,b-2}) \circ \left( \rho^{-1}_{I_{a}} \circ \rho^{-1}_{a-1,b-1}\right).
\end{align*}

We note that each $I_{a-k}$ lies in the $(a-k)$th up-diagonal. Each $\rho^{-1}_{a-\ell,b-\ell}$ to the left of $\rho^{-1}_{I_{a-k}}$ is separated from $I_{a-k}$ by at least one up-diagonal. By Observation~\ref{Obs:togglesCommute} we can then rewrite this as
\[\left(\rho^{-1}_{I_{a-m+2}} \circ \dots \circ \rho^{-1}_{I_{a}}\right) \circ \left(\rho^{-1}_{a-m+1,b-m+1} \circ \dots \circ \rho^{-1}_{a-1,b-1}\right) = {\proid_{a,b}} \circ \mathcal{RSK}_{a,b}. \qedhere\]
\end{proof}

In the case that $(a,b)=(r,s)$, conjugating Lemma~\ref{Lemma:proidRSK} by $\RSK$ gives the following corollary. 
\begin{Cor} \label{cor:omega}
\[\RSK^{-1}_{r,s} \circ \RSK_{r+1,s} = \RSK^{-1} \circ \proid \circ \RSK = \Omega.\]
\end{Cor}

Using Lemma~\ref{Lemma:proidRSK} we can compute any composition $\mathcal{RSK}_{a,b} \circ \mathcal{RSK}_{c,b}^{-1}$. For the following proposition, we state the transposed version since it will be of more use to us.


\begin{Prop}
\label{Prop:RSKCancellation}
Let $a \in [r+1]$ and let $b,c \in [s+1]$. If $c \geq b$, then 
\[\mathcal{RSK}_{a,c} \circ \mathcal{RSK}_{a,b}^{-1} = \RSK_{a,c} \circ \RSK^{-1}_{a,b} = \idpro_{a,c-1} \circ \idpro_{a,c-2} \circ \cdots \circ \idpro_{a,b}\]
In particular, for any $b$ and $c$, $\mathcal{RSK}_{a,c} \circ \mathcal{RSK}_{a,b}^{-1}$ is equivalent to a toggle sequence containing only toggles at elements strictly right of $(a,\min(b,c))$.
\end{Prop}

\begin{proof} If $b \geq c$, consider the inverse:
\[\left( \mathcal{RSK}_{a,c} \circ \mathcal{RSK}_{a,b}^{-1} \right)^{-1} = \mathcal{RSK}_{a,b} \circ \mathcal{RSK}_{a,c}^{-1}.\]
Toggles are involutions, so the inverse of a composition of toggles is the sequence of toggles applied in the reverse order. Thus it suffices to prove the case $b \leq c$. In this case, write $\mathcal{RSK}_{a,c} \circ \mathcal{RSK}_{a,b}^{-1}$ as
\begin{multline*}
(\mathcal{RSK}_{a,c} \circ \mathcal{RSK}_{a,c-1}^{-1}) \circ (\mathcal{RSK}_{a,c-1} \circ \mathcal{RSK}_{a,c-2}^{-1}) \circ \dots \circ (\mathcal{RSK}_{a,b+1} \circ \mathcal{RSK}_{a,b}^{-1})\\
    = \idpro_{a,c-1} \circ \cdots \circ \idpro_{a,b}
\end{multline*}
by the transpose of Lemma~\ref{Lemma:proidRSK}.
\end{proof}

\begin{Ex} 
\label{Ex:doubleRSKExample}
Consider the case of $r = a = b = 3$ and $s = c = 4$. In Figure~\ref{fig:doubleRSKCancelation} the yellow subrectangles of toggles cancel since toggles are involutions. The blue rectangles also cancel since they can commute past the remaining toggles between them.

\begin{figure}
    \centering
    \begin{tikzpicture}[scale = 0.65]
    
        \tikzstyle{fancyR}=[circle, draw, fill=red, inner sep=0pt, minimum width=4.6pt]
        \foreach \z in {0,3,6,9,13}{
        \begin{scope}[rotate = 45, shift = {(\z,-\z)}]
            \draw (0,0) grid (3,2);
            
            \foreach \x in {0,...,3}{
                \foreach \y in {0,...,2}{
                    \node[bigB] at (\x,\y) {};
                }
            }
        \end{scope}
        }
        
        \begin{scope}[rotate = 45, shift={(0,0)}]
        \draw [fill = blue, opacity = 0.3] (-0.3,-0.3)--(0.3,-0.3)--(0.3,0.3)--(-0.3,0.3)--cycle;
        
        \node[fancyR] at (0,0) {};
        \node[fancyR] at (1,0) {};
        
        \draw [-stealth,shift={(1,-1)},ultra thick] (0,0)--(0.7,0.7);
        \end{scope}
        
        \begin{scope}[rotate = 45, shift={(3,-3)}]
        \draw [fill = yellow, opacity = 0.35] (-0.3,-0.3)--(1.3,-0.3)--(1.3,1.3)--(-0.3,1.3)--cycle;
        
        \node[fancyR] at (0,0) {};
        \node[fancyR] at (1,0) {};
        \node[fancyR] at (2,0) {};
        
        \node[fancyR] at (0,1) {};
        \node[fancyR] at (1,1) {};
        \node[fancyR] at (2,1) {};
        
        \draw [-stealth,shift={(1,-1)},ultra thick] (0,0)--(0.7,0.7);
        \end{scope}
        
        \begin{scope}[rotate = 45, shift={(6,-6)}]
        \draw [fill = yellow, opacity = 0.35] (-0.3,-0.3)--(1.3,-0.3)--(1.3,1.3)--(-0.3,1.3)--cycle;
        
        \node[fancyR] at (0,0) {};
        \node[fancyR] at (1,0) {};
        
        \node[fancyR] at (0,1) {};
        \node[fancyR] at (1,1) {};
        
        \draw [stealth-,shift={(1,-1)},ultra thick] (0,0)--(0.7,0.7);
        \end{scope}
        
        \begin{scope}[rotate = 45, shift={(9,-9)}]
        \draw [fill = blue, opacity = 0.3] (-0.3,-0.3)--(0.3,-0.3)--(0.3,0.3)--(-0.3,0.3)--cycle;
        
        \node[fancyR] at (0,0) {};
        
        \draw [stealth-,shift={(1,-1)},ultra thick] (0,0)--(0.7,0.7);
        \end{scope}
        
        \node at (16,1.7) {\Large{=}};
        
        \begin{scope}[rotate = 45, shift={(13,-13)}]
        \node[fancyR] at (1,0) {};
        \node[fancyR] at (2,0) {};
        
        \node[fancyR] at (2,1) {};
        
        \draw [-stealth,shift={(1.8,-1.2)},ultra thick] (0,0)--(-0.7,0.7);
        \end{scope}
        
    \end{tikzpicture}
    \caption{The composition $\RSK_{3,4} \circ \RSK^{-1}_{3,3}$ of toggles in $R = [3] \times [4]$. Toggles shown in red are composed in the order indicated by the arrow beneath the rectangle, from the rightmost rectangle to the leftmost. Toggles in the highlighted regions cancel out.}
    \label{fig:doubleRSKCancelation}
\end{figure}
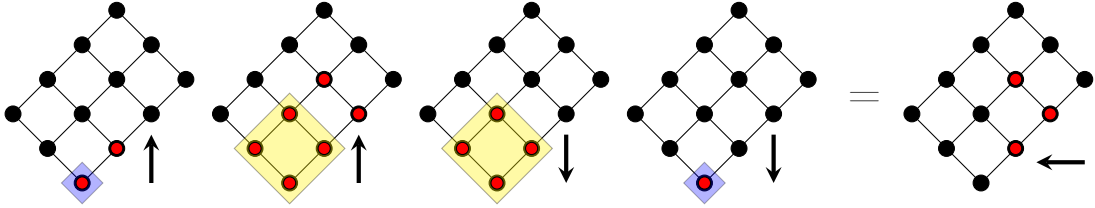
\end{Ex}

We can use Proposition~\ref{Prop:RSKCancellation} to prove that $\RSK_{a,b}^{-1} \circ \RSK_{a,c}$ preserves certain chain statistics.

\begin{Lemma}
\label{Lemma:doubleRSK}
Let $a \in [r]$ and let $b,c \in [s+1]$. Then for all $\ell \leq j \leq  \min(b,c)$,
\[H_{1,\ell}^{r,j}(x;k) = H_{1,\ell}^{r,j}\left(\RSK^{-1}_{a,b} \circ \RSK_{a,c}(x);k\right).\]
\end{Lemma}

On the birational level, Lemma~\ref{Lemma:doubleRSK} can be proven easily using Greene's theorem and the Lindstr\"om-Gessel-Viennot lemma. However, this proof uses subtraction and therefore does not tropicalize to the piecewise-linear case. By using only Greene's theorem and the chain shifting lemma, we can give a proof that also works in the piecewise-linear realm.

\begin{proof}
Without loss of generality, suppose that $b \leq c$. We start with the base case $\ell = 1$. By Proposition~\ref{Prop:RSKCancellation}, $\RSK_{a,c} \circ \RSK^{-1}_{a,b}$ is equivalent to a composition of toggles that are strictly right of $(a,b)$. Similarly, $\RSK_{r,b} \circ \RSK_{a,b}^{-1}$ is equivalent to a composition of toggles that are strictly left of $(a,b)$. Since there is a file of elements separating these, they commute with each other by Observation~\ref{Obs:togglesCommute}. Therefore
\[({\RSK_{a,c}} \circ \RSK_{a,b}^{-1}) \circ ({\RSK_{r,b}} \circ \RSK_{a,b}^{-1}) = ({\RSK_{r,b}} \circ \underline{\RSK_{a,b}^{-1}) \circ (\RSK_{a,c}} \circ \RSK_{a,b}^{-1}).\]
Solving for the underlined quantity gives
\[\RSK_{r,b}^{-1} \circ \RSK_{a,c} \circ \RSK_{a,b}^{-1} \circ \RSK_{r,b} = \RSK_{a,b}^{-1} \circ \RSK_{a,c}.\]

By Greene's Theorem (Theorem~\ref{Th:plGreene}) and Proposition~\ref{miniRSK},
\[H_{1,1}^{r,j}(x;k) = \sum\limits_{m = 0}^{k-1}\RSK_{r,b}(x)_{r-m,j-m}.\]
Recall that $\RSK_{a,c} \circ \RSK^{-1}_{a,b}$ is equivalent to a composition of toggles that are strictly right of $(a,b)$. Since $r \geq a$ and $j \leq b$, $(r-m,j-m)$ is weakly left of $(a,b)$, so these entries are unaffected by $\RSK_{a,c} \circ \RSK^{-1}_{a,b}$. Hence
\begin{align*}
    \sum\limits_{\ell = 0}^{k-1}\RSK_{r,b}(x)_{r-m,j-m} &= \sum\limits_{m = 0}^{k-1} \RSK_{a,c} \circ \RSK^{-1}_{a,b} \circ \RSK_{r,b}(x)_{r-m,j-m}\\
    &= H_{1,1}^{r,j}(\RSK_{r,b}^{-1} \circ \RSK_{a,c} \circ \RSK_{a,b}^{-1} \circ \RSK_{r,b}(x);k)\\
    &= H_{1,1}^{i,j}\left(\RSK^{-1}_{a,b} \circ \RSK_{a,c}(x);k\right)
\end{align*}
by another application of Greene's Theorem.


For the induction step, let $I$ be the order ideal generated by $(a-1,c-1)$ and $(r,b-1)$, and let $J = I \setminus ([a-1] \times [c-1])$. Then
\begin{align*}
    \mathcal{RSK}_{a,b}^{-1} \circ \mathcal{RSK}_{a,c} 
    &= (\rho_{a-1,b-1} \circ \mathcal{RSK}_{a-1,b-1}^{-1}) \circ \rho_{J} \circ \rho_{J}^{-1} \circ (\mathcal{RSK}_{a-1,c-1} \circ \rho_{a-1,c-1}^{-1}).
\end{align*}
Since $J$ lies strictly above the $(a-1)$st up-diagonal, while all toggles in $\mathcal{RSK}_{a-1,b-1}^{-1}$ and $\mathcal{RSK}_{a-1,c-1}$ lie strictly below the $(a-1)$st up-diagonal, by Observation~\ref{Obs:togglesCommute} we can commute the RSK's to the inside to obtain
\begin{multline*}
    \rho_{a-1,b-1} \circ \rho_{J} \circ \mathcal{RSK}_{a-1,b-1}^{-1} \circ  \mathcal{RSK}_{a-1,c-1} \circ \rho_{J}^{-1} \circ \rho_{a-1,c-1}^{-1} \\
    =\rho_{r,b-1} \circ \mathcal{RSK}_{a-1,b-1}^{-1} \circ \mathcal{RSK}_{a-1,c-1} \circ \rho_{I}^{-1}.
\end{multline*}
Applying Lemma~\ref{Lemma:orderidealshifting} with $\rho_{r,b-1}$ and using the inductive hypothesis gives
\begin{align*}
    H_{1,\ell}^{r,j}&\left(\RSK^{-1}_{a,b} \circ \RSK_{a,c}(x);k\right)\\
    &= H_{1,\ell}^{r,j}\left(\phi \circ \rho_{r,b-1} \circ \mathcal{RSK}_{a-1,b-1}^{-1} \circ \mathcal{RSK}_{a-1,c-1} \circ \rho_{I}^{-1} \circ \phi^{-1}(x);k\right) \\
    &= H_{1,\ell-1}^{r,j-1}\left(\phi \circ  \mathcal{RSK}_{a-1,b-1}^{-1} \circ \mathcal{RSK}_{a-1,c-1} \circ \rho_{I}^{-1} \circ \phi^{-1}(x);k\right)\\
    &=H_{1,\ell-1}^{r,j-1}\left(\phi \circ  \rho_I^{-1} \circ \phi^{-1}(x);k\right).
\end{align*}
Applying Lemma~\ref{Lemma:orderidealshifting} again with $\rho_I$ shows that this equals $H_{1,\ell}^{r,j}\left(x;k\right)$, as desired.
\end{proof}

For $\Omega=\RSK^{-1} \circ \proid \circ \RSK$, we have two distinct chain shifting lemmas, which both hold on the piecewise-linear and birational levels. While they are difficult to prove in the coordinates $x$ and $\phi^{-1}(x)$, it is natural to prove them in the RSK coordinates instead. 

\begin{Cor} \label{Cor:shiftsinonedirection} \label{Cor:preservedinonedirection}
Let $x \in \mathbb R^R$, and $k \in \mathbb Z_{\geq 0}$.
\begin{enumerate}[(a)]
\item For all $1 \leq u \leq v \leq r$,
\[H_{u,1}^{v,s}(x;k) = H_{u,1}^{v,s}(\Omega(x);k).\]
\item For all $1 < u \leq v \leq s$,
\[H_{1,u}^{r,v}(x;k) = H_{1,u-1}^{r,v-1}\left( \Omega(x);k\right).\]
\end{enumerate}
\end{Cor}

\begin{proof}
Recall from Corallary~\ref{cor:omega} that $\Omega = \RSK_{r,s}^{-1} \circ \RSK_{r+1,s}$.
Part (a) then follows from the transposed version of Lemma~\ref{Lemma:doubleRSK}. For part (b), note that
\[\mathcal{RSK}_{r,s}^{-1} \circ \mathcal{RSK}_{r+1,s} = \mathcal{RSK}_{r,s}^{-1} \circ \mathcal{RSK}_{r,s-1} \circ \rho^{-1}_{r,s-1}.\]
By Lemma~\ref{Lemma:doubleRSK} we have
\[H_{1,u-1}^{r,v-1}\left( \phi \circ \mathcal{RSK}_{r,s}^{-1} \circ \mathcal{RSK}_{r,s-1} \circ \rho^{-1}_{r,s-1} \circ \phi^{-1}(x);k \right) = H_{1,u-1}^{r,v-1}\left( \phi \circ \rho^{-1}_{r,s-1} \circ \phi^{-1}(x);k \right).\]
The result then follows from Lemma~\ref{Lemma:orderidealshifting}.
\end{proof}

Recall that the $P$-tableau of a labeling $x$ of $[r] \times [s]$ is the set of coordinates of $\RSK(x)$ that lie weakly left of $(r,s)$. Knowing the $P$-tableau of a labeling tells us more chain statistics than just those appearing in $\RSK$. Using $\Omega$, we can prove the following lemma.

\begin{Lemma}
\label{Lemma:samePTableaux}
Let $x$ and $\widetilde{x}$ be labelings of $R$. The following are equivalent:
\begin{enumerate}
    \item $P(x) = P(\widetilde{x})$.
    \item For all $1 \leq u \leq v \leq s$ and $k \in \mathbb{Z}_{\geq 0}$
    \[H_{1,u}^{r,v}(x;k) = H_{1,u}^{r,v}(\widetilde{x};k).\]
\end{enumerate}
\end{Lemma}

\begin{proof}
Obviously (2) implies (1) by Greene's Theorem, and (1) implies the case of (2) where $u=1$. If $u > 1$, then by Corollary~\ref{Cor:shiftsinonedirection}(b),
\[H_{1,u}^{r,v}(x;k) = H_{1,1}^{r,v-u+1}(\RSK^{-1} \circ \left(\proid\right)^{u-1} \circ \RSK(x);k).\]
Toggles in $\left(\proid\right)^{u-1}$ lie strictly left of $(r,s)$, and so they depend only on coordinates weakly left of $(r,s)$. But the coordinates weakly left of $(r,s)$ are in $P(x)$. Consequently if $P(x) = P(\widetilde{x})$, then these $P$-tableaux remain equal after applications of $\proid$.
\end{proof}

In \cite{johnsonliu}, the present authors show that one can use Lemma~\ref{Lemma:orderidealshifting} to compute all coordinates of ${\RSK} \circ \phi \circ \rho^{-1} \circ \phi^{-1}(x)$ except those in the file containing $(r,s)$. 
In this sense, the chain shifting lemma quantifiably ``knows'' most of the information about piecewise-linear rowmotion. For $\RSK^{-1} \circ \RSK_{r+1,s}$, we have the following stronger statement.

\begin{Cor}
The function $\Omega \colon \mathbb{R}^R \to \mathbb{R}^R$ is the unique function satisfying Corollary~\ref{Cor:shiftsinonedirection}. 
\end{Cor}

\begin{proof}
Suppose $f\colon\mathbb{R}^R \to \mathbb{R}^R$ satisfies Corollary~\ref{Cor:shiftsinonedirection}. 
We will show that this information is sufficient to compute ${\RSK} \circ f$.

Let $x \in \mathbb{R}^R$, $j < s$, and $k \in \mathbb{Z}_{\geq 0}$. We seek to compute ${\RSK} \circ f(x)_{r-k,j-k}$. By Theorem~\ref{Th:plGreene},
\[\sum_{t=0}^{k}{\RSK} \circ f(x)_{r-t,j-t} = H_{1,1}^{r,j}(f(x);k+1).\]
By Corollary~\ref{Cor:shiftsinonedirection}(a), this equals $H_{1,2}^{r,j+1}(x;k+1)$,
and so
\[{\RSK} \circ f(x)_{r-k,j-k} = H_{1,2}^{r,j+1}(x;k+1) - H_{1,2}^{r,j+1}(x;k).\]
A similar argument using Corollary~\ref{Cor:preservedinonedirection}(b) computes ${\RSK} \circ f(x)_{i-k,s-k}$ for $i \leq r$.
\end{proof}

\section{The Ehrhart Theory of Northeast Chains} \label{section:fillings}

In this section, we associate to each moon polyomino a rational polytope and show that if two moon polyominoes are equivalent, then their corresponding polytopes have the same Ehrhart series. We cover the relevant definitions and background as needed.

\subsection{Stable set polytopes}
Let $G = (V,E)$ be a graph with vertex set $V$ and edge set $E$. A \emph{clique} in $G$ is a set $C \subseteq V$ such that any two vertices in $C$ are connected by an edge.
\begin{Def}
The \emph{clique constraint stable set polytope} $\qstab(G)$ is the intersection of half spaces in $\mathbb{R}^{V}$ of the form $x_v \geq 0$ for all $v \in V$ and $\sum\limits_{v \in C} x_v \leq 1$ for all (maximal) cliques $C$ of $G$.
\end{Def}
The inequality for a clique $C$ implies the corresponding inequality for any subset $S \subseteq C$ (due to nonnegativity of $x$), so it suffices to let $C$ range over maximal cliques.

While $\qstab(G)$ is always a rational polytope, it is not always a lattice polytope. $\qstab(G)$ is a lattice polytope if and only if $G$ is a \emph{perfect graph}---that is, neither the graph nor its complement contains an induced odd cycle of length at least 5 \cite{grotschellovaszschrijver}.

Let $\m$ be a moon polyomino. We define a graph $G_\m$ on the boxes of $\m$ by drawing an edge from $(u_1, v_1)$ to $(u_2, v_2)$ if $u_1 \leq v_1$, $u_2 \leq v_2$, and $[u_1, u_2] \times [v_1,v_2] \subseteq \m$. In other words, two boxes are connected by an edge if they are comparable in some rectangle poset $R \subseteq \m$. We write $\qm$ for $\qstab(G_\m)$. 

\begin{Ex}
Consider the moon polyomino $\mathcal{M}$ in Figure~\ref{fig:fiveCycle}. In $G_{\mathcal M}$:
\begin{itemize}
    \item There is no edge from $(2,1)$ to $(1,2)$ because these elements are incomparable in the subrectangle containing them. 
    \item There is an edge from $(1,1)$ to $(2,2)$ because these boxes are comparable in the rectangle $[1,2] \times [1,2] \subseteq \mathcal{M}$. Similarly we have an edge from $(2,2)$ to $(3,3)$.
    \item There is no edge from $(1,1)$ to $(3,3)$ because $(1,3) \not \in \mathcal{M}$, so $(1,1)$ and $(3,3)$ do not lie in a common rectangle.
\end{itemize}

Note that due to the induced $5$-cycle shown, $\qstab(\mathcal{M})$ is not a lattice polytope (in fact, $\mathcal M$ is the smallest such moon polyomino). Indeed, $\qstab(\mathcal M)$ contains the rational vertex $v$ where $v_{11}=v_{22}=v_{33}=v_{44}=v_{41}=\frac12$ and all other coordinates are $0$.


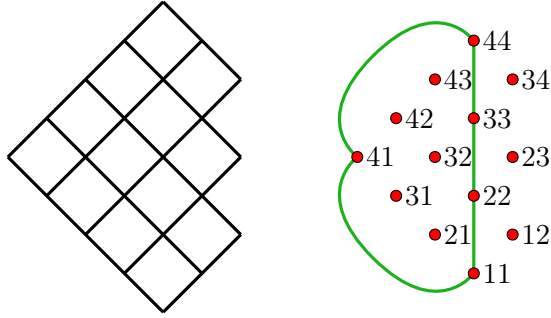
\begin{figure}
    \centering
    \begin{tikzpicture}[scale = 0.73, rotate = 45]
    
    \draw [very thick] (0,0)--(4,0);
    \draw [very thick] (0,-1)--(4,-1);
    \draw [very thick] (0,-2)--(4,-2);
    \draw [very thick] (0,-3)--(3,-3);
    \draw [very thick] (0,-4)--(2,-4);
    
    \draw [very thick] (0,0)--(0,-4);
    \draw [very thick] (1,0)--(1,-4);
    \draw [very thick] (2,0)--(2,-4);
    \draw [very thick] (3,0)--(3,-3);
    \draw [very thick] (4,0)--(4,-2);
    
    \begin{scope}[shift={(4,-4)}]
    \draw [myGreen, very thick] (3.5,-0.5)--(0.5,-3.5);
    \draw [myGreen, very thick] plot [smooth, tension=2] coordinates { (3.5,-0.5) (2,0.6) (0.5,-0.5)};
    \draw [myGreen, very thick] plot [smooth, tension=2] coordinates { (0.5,-3.5) (-0.6,-2) (0.5,-0.5)};
    
    \draw [fill=red] (0.5,-0.5) circle [radius=0.1cm];
    \draw [fill=red] (1.5,-0.5) circle [radius=0.1cm];
    \draw [fill=red] (2.5,-0.5) circle [radius=0.1cm];
    \draw [fill=red] (3.5,-0.5) circle [radius=0.1cm];
    
    \draw [fill=red] (0.5,-1.5) circle [radius=0.1cm];
    \draw [fill=red] (1.5,-1.5) circle [radius=0.1cm];
    \draw [fill=red] (2.5,-1.5) circle [radius=0.1cm];
    \draw [fill=red] (3.5,-1.5) circle [radius=0.1cm];
    
    \draw [fill=red] (0.5,-2.5) circle [radius=0.1cm];
    \draw [fill=red] (1.5,-2.5) circle [radius=0.1cm];
    \draw [fill=red] (2.5,-2.5) circle [radius=0.1cm];
    
    \draw [fill=red] (0.5,-3.5) circle [radius=0.1cm];
    \draw [fill=red] (1.5,-3.5) circle [radius=0.1cm];
    
    \begin{small}
    \node at (0.8,-0.8) {41};
    \node at (1.8,-0.8) {42};
    \node at (2.8,-0.8) {43};
    \node at (3.8,-0.8) {44};
    
    \node at (0.8,-1.8) {31};
    \node at (1.8,-1.8) {32};
    \node at (2.8,-1.8) {33};
    \node at (3.8,-1.8) {34};
    
    \node at (0.8,-2.8) {21};
    \node at (1.8,-2.8) {22};
    \node at (2.8,-2.8) {23};
    
    \node at (0.8,-3.8) {11};
    \node at (1.8,-3.8) {12};
    \end{small}
    \end{scope}
    
    \end{tikzpicture}
    \caption{A moon polyomino with 13 boxes and an induced 5-cycle in its associated graph. This is the smallest moon polyomino for which the associated clique constraint stable set polytope is not a lattice polytope.}
    \label{fig:fiveCycle}
\end{figure}
\end{Ex}

\subsection{Ehrhart theory}
The Ehrhart theory of $\qm$ will be particularly relevant, so we review some definitions and results here. Let $P \subseteq \mathbb{R}^d$ be a convex polytope. The \emph{Ehrhart function} of $P$ is the function $i_P\colon\mathbb{Z}_{\geq 0} \to \mathbb{Z}_{\geq 0}$ defined by
\[i_P(k) = \#\{\mathbb{Z}^d \cap kP\}.\]
When $P$ is a lattice polytope, $i_P(k)$ is a polynomial in $k$ called the \emph{Ehrhart polynomial} of $P$. When the vertices of $P$ are only known to be rational, the Ehrhart function is instead a \emph{quasi-polynomial} --- a function of the form
\[a_d(k)k^d + a_{d-1}(k)k^{d-1} + \cdots +a_0(k)\]
where each $a_i(k)$ is a periodic function.  In either case, the generating function of the Ehrhart function is called the \emph{Ehrhart series} of $P$ and is denoted $\Ehr_P(t)$.

For $\qstab(\mathcal{M})$, lattice points in the $k$th dilate are integer points $x$ satisfying $x_{ij} \geq 0$
for $(i,j) \in \mathcal{M}$ and
$\sum_{(i,j) \in C} x_{ij} \leq k$
for all maximal cliques $C$ in $G_\mathcal{M}$. Since each clique in $G_\mathcal{M}$ forms a chain in some maximal rectangle of $\m$, the lattice points in $k\qstab(\mathcal{M})$ are nonnegative fillings $x$ such that the maximum weight of all weak northeast chains (that is, $H_\m(x;1)$) is bounded above by $k$. Hence we have the following lemma.

\begin{Lemma}
The Ehrhart quasi-polynomial of $QSTAB(\mathcal{M})$ counts the number of nonnegative integer fillings $x$ of $\mathcal{M}$ with $H_\m(x;1) \leq k$.
\end{Lemma}

\subsection{Shifts in moon polyominoes}
Recall from Section~\ref{section:promotionBackground} that Rubey \cite{rubey} constructs a map between fillings of moon polyominoes $\m$ and $\n$ differing by a shift of down-diagonals intersecting a maximal rectangle $R$. Given a filling of $\m$, we obtain a new filling of $\n$ by applying $\Omega$ to the labels of $R$ and shifting certain boxes and labels parallel to the sides of $R$. Using the piecewise-linear versions of $\RSK$ and $\proid$, we can define a piecewise-linear version of $\Omega = \RSK^{-1} \circ \proid \circ \RSK$ as well as the extended map $\Omega_{\m \to \n}$. (If instead $\m$ and $\n$ differ by a shift of up-diagonals with respect to a maximal rectangle $R$, we can similarly define $\Omega_{\m \to \n}$ using $\RSK^{-1} \circ \idpro \circ \RSK$ on $R$.)

For a rectangle $S = [i_1,i_2] \times [j_1,j_2] \subseteq \m$ and $d \in \mathbb{Z}_{\geq 0}$, let
\[H_S(x;d) = H_{i_1,j_1}^{i_2,j_2}(x;d).\]
This map acts on $\qstab(\m)$ according to the following theorem.
\begin{Th} \label{maptheorem} Let $\m$ and $\n$ be moon polyominoes related by a shift of down-diagonals. Then $\Omega_{\m \to \n}$ is a piecewise-linear, volume-preserving, continuous map $\mathbb{R}^\m \to \mathbb{R}^\n$ such that
\begin{enumerate}[(a)]
    \item for all $d \in \mathbb{Z}_{\geq 0}$ and for every maximal rectangle $S$ of $\m$ with corresponding maximal rectangle $S'$ in $\n$, $H_S(x;d) = H_{S'}(\Omega_{\m \to \n}(x);d)$;
    \item for all $k \in \mathbb{Z}_{\geq 0}$, $\Omega_{\m \to \n}(k\qm) = k\qn$; and
    \item for all $k \in \mathbb{Z}_{\geq 0}$, the restriction $\Omega_{\m \to \n}\colon k\qm \cap \mathbb{Z}^\m \to k\qn \cap \mathbb{Z}^\n$ is a bijection.
\end{enumerate}
In particular,
\[\Ehr_{\qm}(t) = \Ehr_{\qn}(t).\]
\end{Th}

\begin{Ex}
Consider the labeling in Figure~\ref{fig:mapExample} and apply $\Omega_{\m \to \n}$, where $\m$ is the moon polyomino on the left and $\n$ on the right.

\begin{figure}
\centering
\begin{tikzpicture}

\begin{scope}[shift={(0,0)}]
\draw (0,0)--(1,1);
\draw (-2,0)--(1,3);

\draw (0,0)--(-1,1);
\draw (2,0)--(-1,3);

\node[wB] at (2,0) {};

\node[wB] at (0,0) {};
\node[wB] at (1,1) {};

\node[wB] at (-2,0) {};
\node[wB] at (-1,1) {};
\node[wB] at (0,2) {};
\node[wB] at (1,3) {};

\node[wB] at (-1,3) {};

\node at (2.35,0) {$c$};

\node at (0.4,0.05) {$b$};
\node at (1.35,1) {$e$};

\node at (-1.65,0) {$a$};
\node at (-0.6,1.05) {$d$};
\node at (0.4,2) {$f$};
\node at (1.35,3.1) {$h$};

\node at (-0.6,3) {$g$};

\draw [ultra thick] (0,0)--(1,1)--(0,2)--(-1,1)--cycle;
\end{scope}

\draw [ultra thick, -stealth] (3,1)--(6,1);
\node at (4.5,1.4) {$\Omega_{\m \to \n}$};

\begin{scope}[shift={(9,0)}]
\draw (0,0)--(1,1);
\draw (-2,0)--(1,3);

\draw (1,1)--(0,2);
\draw (1,-1)--(-2,2);

\node[wB] at (1,-1) {};

\node[wB] at (0,0) {};
\node[wB] at (1,1) {};

\node[wB] at (-2,0) {};
\node[wB] at (-1,1) {};
\node[wB] at (0,2) {};
\node[wB] at (1,3) {};

\node[wB] at (-2,2) {};

\node at (1.35,-1) {$C$};

\node at (0.4,0) {$B$};
\node at (1.35,1) {$E$};

\node at (-1.65,0) {$A$};
\node at (-0.6,1) {$D$};
\node at (0.4,2) {$F$};
\node at (1.35,3) {$H$};

\node at (-1.6,2) {$G$};
\end{scope}

\end{tikzpicture}
\caption{The labeling of moon polyominoes in Example~\ref{Ex:mapExample}, related by an application of $\Omega_{\m \to \n}$.}
\label{fig:mapExample}
\end{figure}
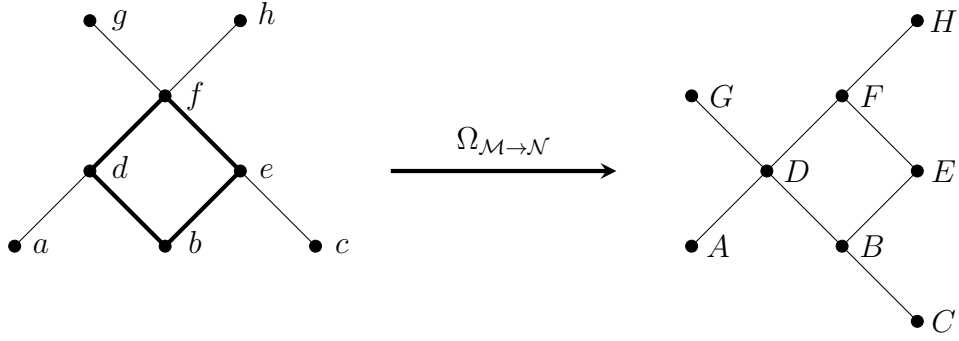
We have
\begin{eqnarray*}
    A&=&a \\
    B&=& \min(b,f)+\max(d,e) - d \\
    C&=&c \\
    D&=& \min(d,e)+\max(b,f) - b \\
    E&=&\min(d,e)+\max(b,f) - f\\
    F&=& \min(b,f)+\max(d,e) - e \\
    G&=&g \\
    H&=&h.
\end{eqnarray*}
Note that since $\max(d,e) \geq d$, $B \geq 0$. By similar arguments, all other coordinates are nonnegative. All coordinates are integers if the original labeling is a lattice point. Consider the following sum:
\[B+D = \min(b,f)+\max(d,e) - d +\min(d,e)+\max(b,f) - b = e+f.\]
So $\Omega$ algebraically shifts the weight $e+f$ downward in the bolded rectangle. Adding in the shifted coordinates $C$ and $G$, we obtain
\[B+C+D+G = c+e+f+g.\]
Intuitively, shifting the weights $C$ and $G$ causes the chain statistics to ``line up'' after applying $\Omega_{\m \to \n}$ so that the weight of this chain is preserved. In general, the maximum weight of all chains in each maximal rectangle is preserved, and so the maximum weights among all northeast chains in both labelings are the same.
\label{Ex:mapExample}
\end{Ex}

Let $x_1$ be a labeling of $[i] \times [s]$ and let $x_2$ be a labeling of $[j] \times [s]$. In many of the remaining results in this paper, we will stack the labelings $x_1$ and $x_2$ to form a labeling $x_1 \concat x_2$  as in the following picture:

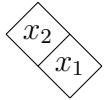
\begin{center}
\begin{tikzpicture}[scale = 0.6]

\begin{scope}[rotate = 45]
\draw[step=1.0,black,thin] (0,0) grid (1,2);
\node at (0.5,0.5) {$x_1$};
\node at (0.5,1.5) {$x_2$};
\end{scope}

\end{tikzpicture}
\end{center}
In other words, we will shift the labeling $x_2$ to be a labeling of $[i+1,i+j] \times [s]$ and concatenate it with $x_1$ to give a new labeling $x_1 \concat x_2$ of $[i+j] \times [s]$. To prove Theorem~\ref{maptheorem} we need the following lemma.

\begin{Lemma}
\label{Lemma:placticMonoid}
If $x_1$ and $\widetilde{x}_1$ are labelings of a rectangle with the same $P$-tableaux, and similarly for $x_2$ and $\widetilde{x}_2$, then $x_1 \concat x_2$ and $\widetilde{x}_1 \concat \widetilde{x}_2$ have the same $P$-tableaux.

\end{Lemma}
On the combinatorial level, this lemma states that the plactic monoid is well-defined as a monoid. Since the proof of Lemma~\ref{Lemma:placticMonoid} using toggles involves evacuation, we defer the proof until Section~\ref{sec:placticMonoid}.

\subsection{Proof of Ehrhart equivalence}

We are now ready to relate the stable set polytopes of $\m$ and $\n$.

\begin{proof}[Proof of Theorem~\ref{maptheorem}]
Let $\mathcal{M}$ and $\mathcal{N}$ differ by a single shift of their down-diagonals with respect to a maximal rectangle $R$. When $\Omega_{\m \to \n}$ acts on $R$, we can express it as a composition of toggles and the transfer map on $R$. Since toggles and the transfer map are piecewise-linear, volume-preserving, and continuous, their composition is as well. The transfer map and toggles also map lattice points to lattice points, so it suffices to prove (a) and (b).

For (a), suppose that we have two maximal rectangles $R$ and $S$ in $\m$ as in Figure~\ref{fig:mapTheoremProofFigure}. We apply $\Omega_{\m \to \n}$ to $R$. Let $S'$ be the image of $S$ in $\n$ under this shifting. Let $x_1,x_2,x_3,$ and $\widetilde{x}_2$ be labelings as in Figure~\ref{fig:mapTheoremProofFigure} (note that the $x_1$ and $x_3$ labels just shift).

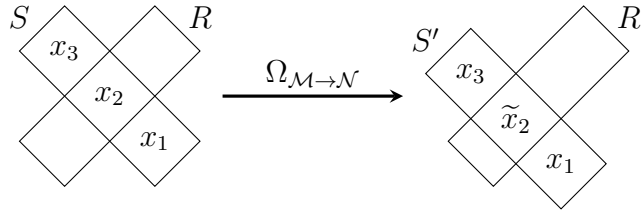
\begin{figure}
\centering
    \begin{tikzpicture}[scale=0.3]
    \node at (-4,9.5) {$S$};
    \node at (4,9.5) {$R$};
    \node at (-2,8) {$x_3$};
    \node at (0,6) {$x_2$};
    \node at (2,4) {$x_1$};
    \draw (2,2)--(4,4)--(-2,10)--(-4,8)--cycle;
    \draw (-2,2)--(-4,4)--(2,10)--(4,8)--cycle;
    
    \node at (9,7) {$\Omega_{\m \to \n}$};
    \draw [-stealth, very thick] (5,6)--(13,6);
    
    \begin{scope}[shift={(19,0)}]
        \node at (-5,8.5) {$S'$};
        \node at (4,9.5) {$R$};
        \node at (-3,7) {$x_3$};
        \node at (-1,5) {$\widetilde{x}_2$};
        \node at (1,3) {$x_1$};
        \draw (1,1)--(3,3)--(-3,9)--(-5,7)--cycle;
        \draw (-2,2)--(-4,4)--(2,10)--(4,8)--cycle;
    \end{scope}
    \end{tikzpicture}
\caption{The action of $\Omega_{\m \to \n}$ on a moon polyomino. Outside of $R$, the labels in the rectangle $S$ shift, while $\Omega$ is applied inside of $R$.}
\label{fig:mapTheoremProofFigure}
\end{figure}

By Corollary~\ref{Cor:shiftsinonedirection}(a) and Lemma~\ref{Lemma:samePTableaux}, $x_2$ and $\widetilde{x}_2$ have the same $P$-tableaux. By Lemma~\ref{Lemma:placticMonoid}, the labelings $x_1 \concat x_2 \concat x_3$ and $x_1\concat \widetilde{x}_2 \concat x_3$ in Figure~\ref{fig:mapTheoremProofFigure} have the same $P$-tableaux. In particular, by Lemma~\ref{Lemma:samePTableaux}
\[H_{S'}(\Omega_{\m \to \n}(x);d) = H_S(x;d).\]
A similar proof holds using Corollary~\ref{Cor:preservedinonedirection}(b) and $Q$-tableaux for rectangles that do not shift under $\Omega_{\m \to \n}$.

For (b), the transfer map and toggles preserve nonnegativity. We know that $\Omega_{\m \to \n}(x)$ satisfies the chain constraint inequalities because this is the special case of (a) when $d = 1$.
\end{proof}

We can relate any equivalent moon polyominoes $\m$ and $\n$ via a series of shifts of up- and down-diagonals (see Figure~\ref{fig:shapeStraightening} for an example). This allows us to define a map $\Omega_{\m \to \n} \colon \mathbb R^{\m} \to \mathbb R^{\n}$ by composing the respective maps for each of these shifts. It is then clear that Theorem~\ref{maptheorem} applies to $\Omega_{\m \to \n}$ as well. (We will prove in Section~\ref{section:commutationProperties} that $\Omega_{\m \to \n}$ is independent of the choice of shifts from $\m$ to $\n$.) 

As mentioned previously, usually the Ehrhart function of a rational polytope is not a polynomial but only a quasi-polynomial. However, one can use the map in Theorem~\ref{maptheorem} to show that the Ehrhart function is a polynomial, that is, it experiences \emph{period collapse} when $\qstab(\m)$ is not a lattice polytope.

\begin{Cor}
The Ehrhart function of $\qstab(\m)$ is a polynomial for any moon polyomino $\m$. In other words, the number of fillings of $\m$ with maximum weight northeast chain of weight at most $k$ is polynomial in $k$. Moreover, this polynomial depends only on $\m$ up to equivalence.
\end{Cor}

\begin{proof}
By Theorem~\ref{maptheorem}, we need only consider moon polyominoes up to equivalence. Any moon polyomino $\m$ is equivalent to some straight partition shape $\lambda$. Consider the poset $P_\lambda$ of boxes of $\lambda$ ordered by $(i,j) \leq (i',j')$ if $i \leq i'$ and $j \leq j'$. The clique constraint inequalities of $\qstab(\lambda)$ are exactly the chain constraint inequalities of the chain polytope of $P_\m$, and so these polytopes are the same. But chain polytopes are lattice polytopes, and so the Ehrhart function is a polynomial.
\end{proof}

\begin{figure}
    \centering
    \begin{tikzpicture}[scale = 0.3]
    \begin{scope}[shift={(0,0)}]
    \draw (1,-1)--(2,0);
    \draw (0,0)--(3,3);
    \draw (-2,0)--(3,5);
    \draw (-3,1)--(2,6);
    \draw (-4,2)--(1,7);
    \draw (-4,4)--(-2,6);
    
    \draw (-2,0)--(-4,2);
    \draw (1,-1)--(-4,4);
    \draw (2,0)--(-3,5);
    \draw (2,2)--(-2,6);
    \draw (3,3)--(0,6);
    \draw (3,5)--(1,7);
    
    \draw [ultra thick](-2,0)--(3,5)--(1,7)--(-4,2)--cycle;
    \end{scope}
    
    \draw [ultra thick, -stealth] (4,3)--(6,3);
    
    \begin{scope}[shift={(12,0)}]
    \draw (0,-2)--(1,-1);
    \draw (-1,-1)--(2,2);
    \draw (-2,0)--(3,5);
    \draw (-3,1)--(2,6);
    \draw (-4,2)--(1,7);
    \draw (-5,3)--(-3,5);
    
    
    \draw (0,-2)--(-5,3);
    \draw (1,-1)--(-4,4);
    \draw (1,1)--(-3,5);
    \draw (2,2)--(-1,5);
    \draw (2,4)--(0,6);
    \draw (3,5)--(1,7);
    
    \draw [ultra thick](-1,-1)--(2,2)--(-1,5)--(-4,2)--cycle;
    \end{scope}
    
    \draw [ultra thick, -stealth] (16,3)--(18,3);
    
    \begin{scope}[shift={(24,0)}]
    \draw (0,-2)--(1,-1);
    \draw (-1,-1)--(4,4);
    \draw (-2,0)--(3,5);
    \draw (-3,1)--(2,6);
    \draw (-4,2)--(-1,5);
    \draw (-5,3)--(-3,5);
    
    
    \draw (0,-2)--(-5,3);
    \draw (1,-1)--(-4,4);
    \draw (1,1)--(-3,5);
    \draw (2,2)--(-1,5);
    \draw (3,3)--(1,5);
    \draw (4,4)--(2,6);
    
    \draw [ultra thick](0,-2)--(1,-1)--(-4,4)--(-5,3)--cycle;
    \end{scope}
    
    \draw [ultra thick, -stealth] (29,3)--(31,3);
    
    \begin{scope}[shift={(37,0)}]
    \draw (0,-2)--(5,3);
    \draw (-1,-1)--(4,4);
    \draw (-2,0)--(3,5);
    \draw (-3,1)--(0,4);
    \draw (-4,2)--(-2,4);
    \draw (-5,3)--(-4,4);
    
    
    \draw (0,-2)--(-5,3);
    \draw (1,-1)--(-4,4);
    \draw (2,0)--(-2,4);
    \draw (3,1)--(0,4);
    \draw (4,2)--(2,4);
    \draw (5,3)--(3,5);
    \end{scope}
    \end{tikzpicture}
    \caption{Equivalence of a moon polyomino to a straight partition shape using diagonal shifts. Shift either the down-diagonals (in the first moon polyomino) or the up-diagonals (in the second and third) intersecting the bolded rectangle to obtain the next shape.}
    \label{fig:shapeStraightening}
\end{figure}
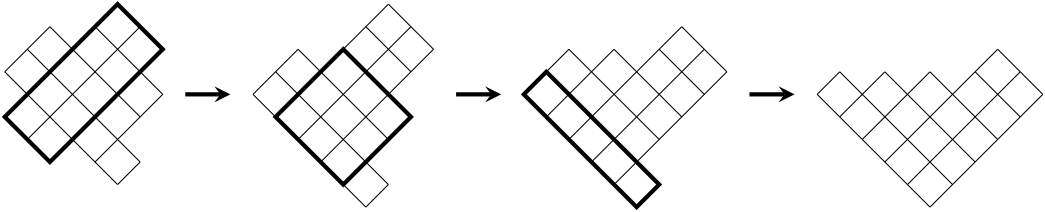

The Ehrhart polynomial of the chain or order polytope of a poset $P$ (also known as the \emph{order polynomial} of $P$) can be calculated using the methods of Stanley \cite{stanley3}. In particular, the normalized volume of this polytope is equal to the number of linear extensions of $P$. In our case, this gives the following immediate corollary.

\begin{Cor}
Let $\m$ be a moon polyomino equivalent to a straight partition shape $\lambda$. Then the normalized volume of $\qm$ is the number of standard Young tableaux of shape $\lambda$.
\end{Cor}

Theorem~\ref{maptheorem} tells us that the maximum weight among all weak northeast chains in $\mathcal{M}$ is preserved by $\Omega_{\m \to \n}$. We close this section by proving the analogous result on strict southeast chains.

\begin{Cor}
Let $\mathcal{M}$ and $\mathcal{N}$ be equivalent moon polyominoes. Then for any $x \in \mathbb{R}_{\geq 0}^\mathcal{M}$, $\Omega_{\m \to \n}$ preserves the maximum length of all strict southeast chains.
\end{Cor}

\begin{proof}
It suffices to consider the case when $\m$ and $\n$ differ by a single shift of down-diagonals. 
Consider the nondecreasing sequence
\[H_\m(x;1) \leq H_\m(x;2) \leq \dots.\]
This sequence stabilizes at some $H_\m(x;k)$, where $k$ is the minimum number such that for all maximal rectangles $R$ of $\m$, the entries of $R$ such that $x_{ij} \neq 0$ can be partitioned into $k$ weak northeast chains. Recall that Dilworth's theorem states that the size of the largest antichain of a poset is the same as the minimum number of chains that partition the elements of the poset. It follows that the length of the longest strict southeast chain is the value of $k$ at which the previous sequence stabilizes. But by Theorem~\ref{maptheorem}(a) we have $H_\m(x;k) = H_\n(\Omega_{\m \to \n}(x);k)$, so the stabilization occurs at the same $k$ for both labelings.
\end{proof}

\section{Evacuation and Commutation}
\label{section:commutationProperties}

In this section, we investigate properties of piecewise-linear (and birational) evacuation and connections to Striker-Williams promotion. We also show that the piecewise-linear (and birational) analogue of $\Omega_{\m \to \n}$ is canonically defined for any equivalent moon polyominoes $\m$ and $\n$ by proving the appropriate commutation property for $\Omega$ applied to different maximal rectangles as shown in the combinatorial case by Rubey \cite{rubey}.

\subsection{Evacuation}
Let $S$ be the induced subposet of $R$ on elements $(i,j)$ weakly left of $(r,s)$ as in Section~\ref{section:promotionBackground}. Let $F^{ij}_k$ denote the intersection of the file indexed by $k$ and the order ideal generated by $(i,j)$. Recall that $\Pro_{ij}$ is the following composition of toggles:
\[\Pro_{ij} = \rho_{F^{ij}_{i-j+1}} \circ \cdots \circ \rho_{F^{ij}_{i-1}}.\]
We have the following definition of evacuation.

\begin{Def}
\emph{Piecewise-linear evacuation} $\evac\colon \mathbb R^S \to \mathbb R^S$ is the composition
\[\evac = \Pro_{r,2} \circ \cdots \circ \Pro_{r,s-1} \circ \Pro_{r,s}.\]
We similarly define $\ei$ on $\mathbb R^R$, as well as $\ie$ by transposition.
\end{Def}
It will also be helpful to define the variant
\[\ei_{i,j} = \proid_{i,2} \circ \proid_{i,3} \circ \cdots \circ \proid_{i,j}\]
and $\ie_{i,j}$ similarly.

The evacuation operation is of interest from a chain shifting perspective because it is a composition of $\proid$ maps, which we know shift chain statistics downward in the poset.

\begin{Prop}
\label{Prop:evacuationChainShifting1}
Let $x \in \mathbb{R}^R$ and let $1 \leq u \leq s$. Then for all $k \in \mathbb{Z}_{\geq 0}$
\[H_{1,s-u+1}^{r,s}(x;k) = H_{1,1}^{r,u}(\RSK^{-1} \circ \ei \circ \RSK(x);k).\]
\end{Prop}

The symmetric statement given by transposing also holds.

\begin{proof}
We induct on $s-u$. The case $u=s$ follows from Theorem~\ref{Th:plGreene} since $\ei$ does not change any entries in the file containing $(r,s)$.

If $s>u$, 
by Corollary~\ref{Cor:shiftsinonedirection}, 
\[H_{1,s-u+1}^{r,s}(x;k) = H_{1,s-u}^{r,s-1}(\Omega(x);k).\]
The inductive hypothesis applied to the restriction of $\Omega(x)$ to the subrectangle $[r] \times [s-1]$ (using Proposition~\ref{miniRSK}) implies
\[H_{1,s-u}^{r,s-1}(\Omega(x);k) =  H_{1,1}^{r,u}({\RSK_{r,s-1}^{-1}} \circ {\ei_{r,s-1}} \circ {\RSK_{r,s-1}} \circ \Omega(x);k).\]
But by Lemma~\ref{Lemma:proidRSK}, $\RSK \circ \RSK_{r,s-1}^{-1} = \idpro_{r,s-1}$ commutes with $\ei_{r,s-1}$, so this equals \begin{align*}
    H_{1,1}^{r,u}({\RSK^{-1}} \circ {\ei_{r,s-1}} \circ {\RSK} \circ \Omega(x);k) &=  H_{1,1}^{r,u}(\RSK^{-1} \circ \ei_{r,s-1} \circ \proid \circ \RSK(x);k)\\
    &= H_{1,1}^{r,u}(\RSK^{-1} \circ \ei \circ \RSK(x);k),
\end{align*}
as desired.
%
\end{proof}

Let $x^*$ be the dual labeling obtained by rotating $x$ by $180^\circ$, so $x^*_{ij} = x_{r+1-i,s+1-j}$. This rotation map satisfies Lemma~\ref{Lemma:evacuationrotates}.

\begin{Lemma}
\label{Lemma:evacuationrotates}
\[\RSK^{-1} \circ \ee \circ \RSK(x) = x^*.\]
\end{Lemma}

\begin{proof}
By Proposition~\ref{Prop:evacuationChainShifting1} we have 
\[H_{1,1}^{r,u}(\RSK^{-1} \circ \ei \circ \RSK(x);k) = H_{1,s-u+1}^{r,s}(x;k) = H_{1,1}^{r,u}(x^*;k).\]
By Lemma~\ref{Lemma:samePTableaux} the $P$-tableaux of $x^*$ and $\RSK^{-1} \circ \ei \circ \RSK(x)$ agree, as does then that of $\RSK^{-1} \circ \ee \circ \RSK(x)$ (since $\ie$ will only affect the $Q$-tableau). Likewise the $Q$-tableaux of $x^*$ and $\RSK^{-1} \circ \ee \circ \RSK$ are also the same, so they must be equal. 
\end{proof}
Note that Lemma~\ref{Lemma:evacuationrotates} implies that $\evac$ is an involution (this can also be proved easily from the toggle definition).

We can also prove a chain shifting lemma, which we state below for $\ei$ (an analogous statement holds for $\ie$).

\begin{Lemma}
\label{Lemma:evacuationChainShifting2}
Let $x \in \mathbb{R}^R$ and let $k \in \mathbb{Z}_{\geq 0}$. 
\begin{enumerate}[(a)]
    \item If $1 \leq u \leq v \leq r$, then
    \[H_{u,1}^{v,s}(x;k) = H_{u,1}^{v,s}(\RSK^{-1} \circ \ei \circ \RSK(x);k).\]
    \item If $1 \leq u \leq v \leq s$, then
    \[H_{1,u}^{r,v}(x;k) = H_{1,s+1-v}^{r,s+1-u}(\RSK^{-1} \circ \ei \circ \RSK(x);k).\]
\end{enumerate}
\end{Lemma}

\begin{proof}
For (a), $\RSK^{-1} \circ \ei \circ \RSK$ does not change the $Q$-tableau of $x$, so the result follows by Lemma~\ref{Lemma:samePTableaux} (for the $Q$ tableaux).
By Lemma~\ref{Lemma:evacuationrotates}, $\RSK^{-1} \circ \ei \circ \RSK(x)$ and $x^*$ have the same $P$-tableaux. Then (b) follows from Lemma~\ref{Lemma:samePTableaux} since $H_{1,u}^{r,v}(x;k) = H_{1,s+1-v}^{r,s+1-u}(x^*;k)$.
\end{proof}

\subsection{Striker-Williams Promotion}
In this subsection we use properties of evacuation to prove a chain shifting lemma for Striker-Williams promotion.

\begin{Def}
Let $F_k$ denote the $k$th file of $R$. Then \emph{Striker-Williams promotion} is the map
\[\swpro = \rho_{F_{1-s}} \circ \rho_{F_{2-s}} \circ \dots \circ \rho_{F_{r-1}}.\]
\end{Def}

Striker and Williams introduce this map for \emph{rc-posets} in \cite{strikerwilliams}. Visually, we toggle at every element in the files of $R$ from the leftmost file to the rightmost. Striker-Williams promotion can be written in the $\proid$ notation as
\[\swpro = (\idpro)^{-1} \circ \rho_{F_{r-s}} \circ \proid.\]
We can use this factorization to prove the following chain shifting lemma for $\swpro$.

\begin{Lemma}
\label{Lemma:SWProChainShifting}
Let $R = [r] \times [s]$ and let $k \in \mathbb{Z}_{\geq 0}$.
\begin{enumerate}[(a)]
    \item If $2 \leq u \leq v \leq s$, then
    \[H_{1,u}^{r,v}(x;k) = H_{1,u-1}^{r,v-1}(\RSK^{-1} \circ \swpro \circ \RSK(x);k).\]
    \item If $1 \leq u \leq v \leq r-1$, then
    \[H_{u,1}^{v,s}(x;k) = H_{u+1,1}^{v+1,s}(\RSK^{-1} \circ \swpro \circ \RSK(x);k).\]
\end{enumerate}
\end{Lemma}

\begin{proof}
We first prove (a). Let $\widetilde{x} = \RSK^{-1} \circ \swpro \circ \RSK(x)$, and let $\pi(x)$ denote the coordinate projection of $x$ to coordinates indexed by $[r] \times [s-1]$. By Theorem~\ref{Th:plGreene}, $P(\pi(x))$ is the collection of labels of $P(x)$ strictly left of $(r,s)$. Since $\idproinv \circ \rho_{F_{r,s}}$ only changes coordinates weakly right of $(r,s)$,
\[P(\pi(\widetilde{x})) = P(\pi(\Omega(x))).\]
By Lemma~\ref{Lemma:samePTableaux}
\[H_{1,u-1}^{r,v-1}(\widetilde{x};k) = H_{1,u-1}^{r,v-1}(\Omega(x);k).\]
Then (a) follows from Corollary~\ref{Cor:shiftsinonedirection}.

For (b), note that $x = \RSK^{-1} \circ \swpro^{-1} \circ \RSK(\widetilde{x})$. Let $x^T$ denote the transpose of $x$ and similarly for $\widetilde{x}^T$. Since $\swpro^{-1}$ is given by toggling along files from the rightmost to the leftmost, we have $x^T = \RSK^{-1} \circ \swpro \circ \RSK(\widetilde{x}^{\,T})$. Then (b) becomes
\[H_{1,u}^{s,v}(x^T;k) = H_{1,u+1}^{s,v+1}(\widetilde{x}^{\,T};k),\]
which follows directly from (a).
\end{proof}

For the remainder of this subsection, we reprove Lemma~\ref{Lemma:SWProChainShifting} as a consequence of the fact that $\rho^{-1}$ and $\swpro$ are conjugate to each other in the group generated by all toggles in $R$. Striker and Williams \cite{strikerwilliams} define a composition of toggles $D$ for rc-posets such that
\[\swpro = D \circ \rho^{-1} \circ D^{-1}.\]
We define a similar element here. Let
\[E = \rho_{1,s}^{-1} \circ \rho_{2,s}^{-1} \circ \dots \circ \rho_{r-1,s}^{-1}.\]

\begin{Prop}
\label{Prop:conjugate}
\[\swpro = E \circ \rho^{-1} \circ E^{-1}.\]
\end{Prop}

\begin{proof}
Let $\swpro_{i,j}$ and $E_{i,j}$ denote the composition of toggles on labelings of $R$ analogous to $\swpro$ and $E$ on the order ideal $[i] \times[j]$. Observe that
\[\swpro_{1,s} = E_{1,s} \circ \rho^{-1}_{1,s} \circ E^{-1}_{1,s}\]
since $E_{1,s}$ contains no toggles. By induction suppose that
\[\swpro_{i-1,s} = E_{i-1,s} \circ \rho^{-1}_{i-1,s} \circ E^{-1}_{i-1,s}.\]
If $I = \{i\} \times [s]$, then
\begin{align*}
    E_{i,s} \circ \rho^{-1}_{i,s} \circ E^{-1}_{i,s} &= E_{i,s} \circ \rho^{-1}_I \circ \rho^{-1}_{i-1,s} \circ E_{i,s}^{-1} \\
    &= E_{i,s} \circ \rho^{-1}_I \circ E_{i-1,s}^{-1}
\end{align*}
All toggles in $\rho^{-1}_I$ lie in the $i$th up-diagonal and all toggles in $E_{i-1,s}^{-1}$ lie strictly below the $(i-1)$st up-diagonal. By Observation~\ref{Obs:togglesCommute} these toggle sequences commute, yielding
\[E_{i,s} \circ E_{i-1,s}^{-1} \circ \rho^{-1}_I.\]
We peel $\rho^{-1}_{i-1,s}$ off of $E_{i,s}$ to obtain
\[E_{i-1,s} \circ \rho^{-1}_{i-1,s} \circ E_{i-1,s}^{-1} \circ \rho^{-1}_I.\]
By induction this equals
\[{\swpro_{i-1,s}} \circ \rho^{-1}_I = \swpro_{i,s}.\qedhere\]
\end{proof}

Throughout this paper we use $\mathcal{RSK}_{r,s}$ to map from the setting where one applies rowmotion to the setting where one applies promotion. One desirable quality in a toggle sequence relating rowmotion and Striker-Williams promotion by conjugation is that it has a natural description in terms of RSK. We have the following lemma regarding $E$.

\begin{Lemma}
\label{Lemma:Einmapsweknow}
\[E = {\ie} \circ \mathcal{RSK}.\]
\end{Lemma}

\begin{proof}
First write
\begin{align*}
    {\ie} \circ \mathcal{RSK} &= {\ie_{r-1,s}} \circ {\idpro} \circ \mathcal{RSK}_{r,s}\\ &= {\ie_{r-1,s}} \circ \mathcal{RSK}_{r,s+1}
\end{align*}
by the transpose of Lemma~\ref{Lemma:proidRSK}. Then by the transpose of Lemma~\ref{Lemma:proidRSK} again,
\begin{align*}
{\ie_{r-1,s}} \circ \mathcal{RSK}_{r,s+1} &= {\ie_{r-2,s}} \circ {\idpro_{r-1,s}} \circ \mathcal{RSK}_{r,s+1}\\
    &={\ie_{r-2,s}} \circ{\mathcal{RSK}_{r-1,s+1}} \circ {\mathcal{RSK}_{r-1,s}^{-1}} \circ \mathcal{RSK}_{r,s+1} \\
    &={\ie_{r-2,s}} \circ{\mathcal{RSK}_{r-1,s+1}} \circ \rho_{r-1,s}^{-1}.
\end{align*}
Repeating this argument, we find that
\begin{align*}
    {\ie_{r-1,s}} \circ \mathcal{RSK}_{r,s+1}
    &= {\ie_{r-2,s}} \circ \mathcal{RSK}_{r-1,s+1} \circ \rho^{-1}_{r-1,s}\\
    &={\ie_{r-3,s}} \circ \mathcal{RSK}_{r-2,s+1} \circ \rho^{-1}_{r-2,s} \circ \rho^{-1}_{r-1,s}\\
    &\phantom{=}\vdots\\
    &= \rho^{-1}_{1,s} \circ \dots \circ \rho^{-1}_{r-1,s}\\ &= E.\qedhere
\end{align*}
\end{proof}

Since $\swpro = E \circ \rho^{-1} \circ E^{-1}$, and $E$ and $\rho^{-1}$ have chain shifting lemmas, we now reprove the previous chain shifting lemma for $\swpro$.


\begin{proof}[Proof of Lemma~\ref{Lemma:SWProChainShifting}]
For (b), if $1 \leq u \leq v \leq r-1$, then by Proposition~\ref{Prop:conjugate} and Lemma~\ref{Lemma:Einmapsweknow},
\begin{align*}
    &H_{u+1,1}^{v+1,s}(\RSK^{-1} \circ \swpro \circ \RSK(x);k) \\
    &= H_{u+1,1}^{v+1,s}(\RSK^{-1} \circ \ie \circ \mathcal{RSK}_{r,s} \circ \rho^{-1} \circ \mathcal{RSK}^{-1}_{r,s} \circ \ie \circ \RSK(x);k).
\end{align*}
By Lemma~\ref{Lemma:evacuationChainShifting2} this expression simplifies to
\[H_{r-v,1}^{r-u,s}(\phi \circ \rho^{-1} \circ \mathcal{RSK}^{-1}_{r,s} \circ \ie \circ \RSK(x);k).\]
By Lemma~\ref{Lemma:orderidealshifting} this equals
\[H_{r-v+1,1}^{r-u+1,s}(\RSK \circ \ie \circ \RSK(x);k),\]
and so (b) follows by another application of Lemma~\ref{Lemma:evacuationChainShifting2}. A similar proof yields (a).
\end{proof}

\begin{Ex} Proposition~\ref{Prop:conjugate} and Lemma~\ref{Lemma:Einmapsweknow} imply that the diagram in Figure~\ref{fig:SWProChainShifting} commutes. All red chains in Figure~\ref{fig:SWProChainShifting} have the same weight, which is the maximum weight in each blue rectangle. Then $\swpro$ shifts the maximum weight among chains in interval $[i,j] \times [s]$ upward in the poset and shifts weight of chains in intervals $[r] \times [i,j]$ downward in the poset. Alternatively, we can think of this upward shifting by first applying $\RSK^{-1} \circ \ie \circ \RSK$, which reflects the rectangle along the axis perpendicular to the direction we shift. Note that chains themselves do not reflect, only the region they lie in. Then applying $\phi \circ \rho^{-1} \circ \phi^{-1}$ shifts the maximum weight of chains downward in the poset. Reflecting again transforms this downward shift into an upward shift. 

\begin{figure}
    \centering
    \begin{tikzpicture}
    \begin{scope}[rotate=45]
    \draw (0,0) grid (2,4);
    
    \draw [blue, fill = blue, opacity = 0.17] (-0.2,0.8)--(2.2,0.8)--(2.2,2.2)--(-0.2,2.2)--cycle;
    
    \draw [red,ultra thick] (0,1)--(1,1)--(1,2)--(2,2);
    
    \foreach \x in {0,...,2}{
    \foreach \y in {0,...,4}{
    \node[bigB] at (\x,\y) {};
    }
    }
    
    \node[wR] at (0,1) {};
    \node[wR] at (1,1) {};
    \node[wR] at (1,2) {};
    \node[wR] at (2,2) {};
    
    \node at (0.35,-0.35) {0.5};
    \node at (1.3,-0.3) {0};
    \node at (2.35,-0.35) {0.2};
    
    \node at (0.35,0.65) {0.1};
    \node at (1.4,0.6) {0.05};
    \node at (2.4,0.6) {0.05};
    
    \node at (0.25,1.75) {0};
    \node at (1.42,1.58) {0.15};
    \node at (2.25,1.75) {0};
    
    \node at (0.35,2.65) {0.2};
    \node at (1.4,2.6) {0.05};
    \node at (2.35,2.65) {0.1};
    
    \node at (0.4,3.6) {0.15};
    \node at (1.25,3.75) {0};
    \node at (2.25,3.75) {0};
    \end{scope}
    
    \node at (3.7,2.52) {$\phi \circ \rho^{-1} \circ \phi^{-1}$};
    \draw [-stealth, ultra thick] (2.7,2.12)--(4.7,2.12);
    
    \begin{scope}[rotate=45, shift = {(6,-6)}]
    \draw (0,0) grid (2,4);
    
    \draw [blue, fill = blue, opacity = 0.17] (-0.2,-0.2)--(2.2,-0.2)--(2.2,1.2)--(-0.2,1.2)--cycle;
    
    \draw [red,ultra thick] (0,0)--(1,0)--(1,1)--(2,1);
    
    \foreach \x in {0,...,2}{
    \foreach \y in {0,...,4}{
    \node[bigB] at (\x,\y) {};
    }
    }
    
    \node[wR] at (0,0) {};
    \node[wR] at (1,0) {};
    \node[wR] at (1,1) {};
    \node[wR] at (2,1) {};
    
    \node at (0.25,-0.25) {0};
    \node at (1.4,-0.4) {0.15};
    \node at (2.4,-0.4) {0.05};
    
    \node at (0.35,0.75) {0};
    \node at (1.35,0.65) {0.1};
    \node at (2.4,0.6) {0.05};
    
    \node at (0.35,1.65) {0.2};
    \node at (1.25,1.75) {0};
    \node at (2.4,1.6) {0.15};
    
    \node at (0.4,2.6) {0.05};
    \node at (1.35,2.65) {0.1};
    \node at (2.25,2.75) {0};
    
    \node at (0.25,3.75) {0};
    \node at (1.25,3.75) {0};
    \node at (2.4,3.6) {0.05};
    \end{scope}
    
    \node at (1.2,-2) {\begin{tabular}{c}$\RSK^{-1} \circ \ie$  \\ $\circ \RSK$\end{tabular}};
    \draw [-stealth, ultra thick] (-0.71,-1)--(-0.71,-3);
    
    \begin{scope}[shift={(8.49,0)}]
        \node at (-2.62,-2) {\begin{tabular}{c}$\RSK^{-1} \circ \ie$  \\ $\circ \RSK$\end{tabular}};
        \draw [-stealth, ultra thick] (-0.71,-1)--(-0.71,-3);
    \end{scope}
    
    \begin{scope}[rotate=45,shift={(-5.75,-5.75)}]
    \draw (0,0) grid (2,4);
    
    \draw [blue, fill = blue, opacity = 0.17] (-0.2,1.8)--(2.2,1.8)--(2.2,3.2)--(-0.2,3.2)--cycle;
    
    \draw [red,ultra thick] (0,2)--(1,2)--(1,3)--(2,3);
    
    \foreach \x in {0,...,2}{
    \foreach \y in {0,...,4}{
    \node[bigB] at (\x,\y) {};
    }
    }
    
    \node[wR] at (0,2) {};
    \node[wR] at (1,2) {};
    \node[wR] at (1,3) {};
    \node[wR] at (2,3) {};
    
    \node at (0.25,-0.25) {0};
    \node at (1.25,-0.25) {0};
    \node at (2.4,-0.4) {0.15};
    
    \node at (0.4,0.6) {0.15};
    \node at (1.25,0.75) {0};
    \node at (2.35,0.65) {0.2};
    
    \node at (0.4,1.6) {0.05};
    \node at (1.35,1.65) {0.1};
    \node at (2.25,1.75) {0};
    
    \node at (0.4,2.6) {0.05};
    \node at (1.4,2.6) {0.15};
    \node at (2.25,2.75) {0};
    
    \node at (0.35,3.65) {0.7};
    \node at (1.25,3.75) {0};
    \node at (2.25,3.75) {0};
    \end{scope}

    \node at (3.7,-5.51) {\begin{tabular}{c}$\RSK^{-1} \circ \swpro$ \\ $\circ \RSK$\end{tabular}};
    \draw [-stealth, ultra thick] (2.7,-6.01)--(4.7,-6.01);
    
    \begin{scope}[rotate=45,shift={(0.25,-11.75)}]
    \draw (0,0) grid (2,4);
    
    \draw [blue, fill = blue, opacity = 0.17] (-0.2,2.8)--(2.2,2.8)--(2.2,4.2)--(-0.2,4.2)--cycle;
    
    \draw [red,ultra thick] (0,3)--(1,3)--(1,4)--(2,4);
    
    \foreach \x in {0,...,2}{
    \foreach \y in {0,...,4}{
    \node[bigB] at (\x,\y) {};
    }
    }
    
    \node[wR] at (0,3) {};
    \node[wR] at (1,3) {};
    \node[wR] at (1,4) {};
    \node[wR] at (2,4) {};
    
    \node at (0.25,-0.25) {0};
    \node at (1.4,-0.4) {0.05};
    \node at (2.25,-0.25) {0};
    
    \node at (0.25,0.75) {0};
    \node at (1.4,0.6) {0.15};
    \node at (2.25,0.75) {0};
    
    \node at (0.4,1.6) {0.15};
    \node at (1.25,1.75) {0};
    \node at (2.35,1.65) {0.2};
    
    \node at (0.4,2.6) {0.05};
    \node at (1.35,2.65) {0.1};
    \node at (2.25,2.75) {0};
    
    \node at (0.4,3.6) {0.05};
    \node at (1.4,3.6) {0.05};
    \node at (2.35,3.65) {0.1};
    \end{scope}
    
    \end{tikzpicture}
    \caption{A commuting diagram showing how the chain shifting lemma for $\rho^{-1}$ induces the chain shifting lemma for $\swpro$.}
    \label{fig:SWProChainShifting}
\end{figure}
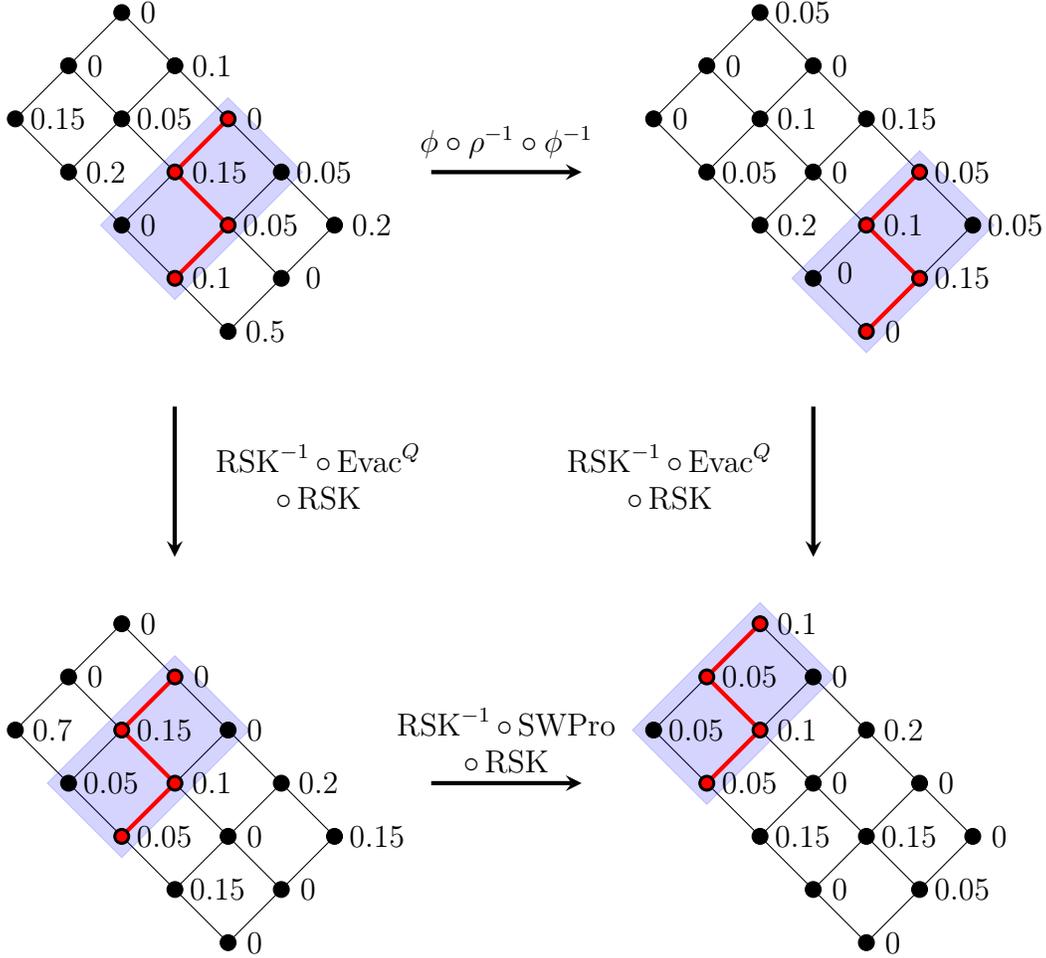

\end{Ex}

\subsection{Proof of Lemma~\ref{Lemma:placticMonoid}} \label{sec:placticMonoid}
In this subsection we prove Lemma~\ref{Lemma:placticMonoid}. For a labeling $x$ of $[r] \times [s]$, let $\overline{Q}(x)$ denote the coordinates of the $Q$-tableau that lie strictly right of $(r,s)$. Given a labeling $(x_1 \concat \dots \concat x_n)$ and some $i \in [n]$, we say that a function $f$ is \emph{independent} of $\overline{Q}(x_i)$ if $f$ can be expressed as a function of $x_1,\dots,x_{i-1},x_{i+1},\dots,x_n,$ and $P(x_i)$.

\begin{Prop}
\label{Prop:RSKcoordinates}
Let $x = x_1 \concat x_2$ be a labeling of $R$, where $(i,s)$ is the maximum element labeled by $x_1$.  Then the coordinates of $\RSK_{i,s}(x)$ weakly left of $(i,s)$ can be expressed invertibly in terms of $P(x_1)$ and $x_2$ (independent of $\overline{Q}(x_1)$).
\end{Prop}

\begin{proof}
Let $(k,j) \in R$ with $(k,j)$ weakly left of $(i,s)$. If $k \leq i$, then $(k,j) \in [i] \times [s]$ and by Proposition~\ref{miniRSK}
\[\RSK_{i,s}(x)_{k,j} = P(x_1)_{k,j}.\]
If $k > i$, then by Proposition~\ref{miniRSK}
\[\RSK_{i,s}(x)_{k,j} = \phi^{-1}(x)_{k,j} = \max\limits_{\ell \in [1,j]} \left(\phi^{-1}(x)_{i,\ell} + H_{i+1,\ell}^{k,j}(x;1)\right).\]
Note that $\phi^{-1}(x)_{i,\ell}$ is a coordinate of $P(x_1)$ and $H_{i+1,\ell}^{k,j}(x;1)$ depends only on $x_2$.

We can easily invert this procedure to recover $P(x_1)$ (trivial) and $x_2$ by computing \[\phi(\RSK_{i,s}(x))_{k,j} = \RSK_{i,s}(x)_{k,j} - \max \left\{\RSK_{i,s}(x)_{k,j-1}, \RSK_{i,s}(x)_{k-1,j}\right\}\] for $k > i$.
\end{proof}


We are now ready to prove that $P(x_1 \concat x_2)$ only depends on $P(x_1)$ and $P(x_2)$.

\begin{proof}[Proof of Lemma~\ref{Lemma:placticMonoid}]

Let $R = [r] \times [s]$, let $x_1$ be a labeling of $[i] \times [s]$ for some $i<r$, and let $x_2$ be a labeling of $[r-i] \times [s]$. It suffices to show that $P(x_1 \concat x_2)$ is independent of $\overline{Q}(x_1)$ and $\overline{Q}(x_2)$.

We first show that $P(x_1 \concat x_2)$ is independent of $\overline{Q}(x_1)$. Let $y = \RSK_{i,s}(x_1 \concat x_2)$. By Proposition~\ref{Prop:RSKcoordinates} $y_{k,j}$ is a function of $P(x_1)$ and $x_2$ when $(k,j)$ is weakly left of $(i,s)$. Note that
\[\RSK(x) = \RSK \circ \RSK^{-1}_{i,s}(y).\]
By Proposition~\ref{Prop:RSKCancellation}, $\RSK \circ \RSK^{-1}_{i,s}$ is equivalent to a toggle sequence that only contains toggles strictly left of $(i,s)$. The result of each such toggle depends only on entries weakly left of $(i,s)$, which are all functions of $P(x_1)$ and $x_2$.

We now show that $P(x_1 \concat x_2)$ is independent of $\overline{Q}(x_2)$. Let 
$x^*_2 \concat x^*_1 = (x_1 \concat x_2)^*$. By Lemma~\ref{Lemma:evacuationrotates},\[P(x_1 \concat x_2) = \text{Evac}(P(x^*_2 \concat x^*_1)),\]
so it suffices to show that $P(x^*_2 \concat x^*_1)$ is a function of $x_1$ and $P(x_2)$. By the previous argument, we know $P(x^*_2 \concat x^*_1)$ is a function of $P(x^*_2)$ and $x^*_1$. By Lemma~\ref{Lemma:evacuationrotates}, $P(x^*_2) = \evac(P(x_2))$ is a function of $P(x_2)$, and clearly $x^*_1$ is a function of $x_1$.
\end{proof}

\subsection{Commutation}
In this subsection, we prove a commutation theorem for maps on labelings of moon polyominoes. This theorem will apply not only to maps like $\Omega$ applied to a maximal rectangle of a moon polyomino, but also to the map $\phi \circ \rho^{-1} \circ \phi^{-1}$ applied as follows. 

Let $\m$ be a moon polyomino with maximal rectangle $R$. Suppose that the equivalent moon polyomino $\n$ is formed from $\m$ by shifting boxes outside $R$ downward parallel to both sets of sides of $R$. (This differs from the extension of $\Omega$, where we shift boxes parallel to only one side---this reflects the differences in the chain shifting lemmas of $\phi \circ \rho^{-1} \circ \phi^{-1}$ and $\Omega$.) We can extend $\phi \circ \rho^{-1} \circ \phi^{-1}$ to a map $\mathbb{R}^{\m} \to \mathbb{R}^{\n}$ by applying $\phi \circ \rho^{-1} \circ \phi^{-1}$ to the coordinates of $R$ and shifting all other coordinates appropriately.

\begin{Th}
\label{Th:commutation}
Let $\psi_1$ and $\psi_2$ be maps in the set
\[\left\{ \RSK^{-1} \circ \proid \circ \RSK, \RSK^{-1} \circ \idpro \circ \RSK, \phi \circ \rho^{-1} \circ \phi^{-1} \right\}\]
acting on distinct maximal rectangles $R_1$ and $R_2$ of some moon polyomino, respectively (shifting labels outside the rectangle appropriately). Then
\[\psi_1 \circ \psi_2 = \psi_2 \circ \psi_1.\]
\end{Th}

We prove Theorem~\ref{Th:commutation} at the end of this subsection. In particular, note that (disallowing $\phi \circ \rho^{-1} \circ \phi^{-1}$) this implies that the map $\Omega_{\m \to \n}$ described in Section~\ref{section:fillings} between any equivalent moon polyominoes $\m$ and $\n$ is independent of the choice and order of shifts from one to the other.

To prove Theorem~\ref{Th:commutation}, we will need some preliminary results about how applying various maps depends on the $Q$-tableaux of subrectangles.

\begin{Prop}
\label{Prop:evacCoordinates}
Let $x = x_1 \concat x_2 \concat x_3$ be a labeling of $R$ such that $(i,s)$ and $(j,s)$ are the maximum elements labeled by $x_1$ and $x_2$, respectively. If $\widetilde{x} = \RSK^{-1} \circ \ie_{j,s} \circ \RSK(x) = \widetilde x_2 \concat \widetilde x_1 \concat \widetilde x_3$ (where $\widetilde x_i$ and $x_i$ label rectangles of the same size), then:
\begin{enumerate}[(a)]
    \item $Q(\widetilde{x}_2) = \evac(Q(x_2))$, and
    \item the coordinates of $\RSK(\widetilde x)$ weakly left of $(j,s)$ can be expressed invertibly in terms of $x_1$, $P(x_2)$, and $x_3$ (independent of $\overline Q(x_2)$).
\end{enumerate}
\end{Prop}

\begin{proof}
For (a), by Theorem~\ref{Th:plGreene} we know $Q(x_1 \concat x_2)$ is the restriction of $Q(x)$ to labels weakly right of $(j,s)$. Applying $\ie_{j,s}$, we find that $Q(\widetilde x_2 \concat \widetilde x_1) = \evac(Q(x_1 \concat x_2)),$ which equals $Q(x_2^* \concat x_1^*)$ by Lemma~\ref{Lemma:evacuationrotates}. Then restricting to the labels weakly right of $(j-i,s)$ gives $Q(\widetilde x_2) = Q(x_2^*) = \evac Q(x_2)$.

For (b), write \[\RSK(\widetilde x)=\ie_{j,s} \circ \RSK(x) = {\ie_{j,s}} \circ T \circ \RSK_{j,s}(x),\] where by Proposition~\ref{Prop:RSKCancellation} $T=\RSK \circ \RSK_{j,s}^{-1}$ is equivalent to a composition of toggles that lie strictly left of $(j,s)$. By Proposition~\ref{Prop:RSKcoordinates}, coordinates of $\RSK_{j,s}(x)$ that lie weakly left of $(j,s)$ depend only on $P(x_1 \concat x_2)$ and $x_3$. By Lemma~\ref{Lemma:placticMonoid} $P(x_1 \concat x_2)$ is a function of $P(x_1)$ and $P(x_2)$, so these coordinates are independent of $\overline Q(x_2)$. Since each toggle in $T$ lies strictly left of $(j,s)$, the result of each of these toggles remains independent of $\overline Q(x_2)$. Since $\ie_{j,s}$ does not change labels weakly left of $(j,s)$, we conclude that the coordinates of $\RSK(\widetilde x)$ that are weakly left of $(j,s)$ are independent of $\overline Q(x_2)$.

Now consider labels that are both weakly right of $(j,s)$ and weakly left of $(j-i,s)$. All coordinates weakly right of $(j,s)$ form the $Q$-tableau $\evac(Q(x_1 \concat x_2)) = Q(x_2^* \concat x_1^*)$. By Proposition~\ref{Prop:RSKcoordinates}, the coordinates of $Q(x_2^* \concat x_1^*)$ that are weakly left of $(j-i,s)$ are functions of $P(x_2^*)$ and $x_1^*$. But both $P(x_2^*)=\evac(P(x_2))$ and $x_1^*$ are functions of $P(x_2)$ and $x_1$ and hence independent of $\overline Q(x_2)$.

This procedure can be inverted using Proposition~\ref{Prop:RSKcoordinates}: $P(x_1 \concat x_2)$ and $x_3$ can be recovered from the coordinates of $\RSK(\widetilde x)$ weakly left of $(j,s)$ as in the first half of the argument, and combining $\evac(P(x_1 \concat x_2)) = P(x_2^* \concat x_1^*)$ with the remaining coordinates weakly left of $(j-i,s)$ yields $P(x_2^*)$ and $x_1^*$ as in the second half.
\end{proof}

\begin{Lemma}
\label{Lemma:independenceOfQ}
Let $x = (x_1 \concat x_2 \concat x_3)$ be a labeling of $R$ such that $(i,s)$ and $(j,s)$ are the maximum elements labeled by $x_1$ and $x_2$ respectively. Suppose that we have one of the following scenarios:
\begin{enumerate}[(a)]
    \item Let $T$ be a composition of toggles strictly left of $(j,s)$, and let \[\widetilde{x} = (\widetilde{x}_1 \concat \widetilde{x}_2 \concat \widetilde{x}_3) = {\RSK^{-1}} \circ T \circ \RSK(x),\] where $(i,s)$ and $(j,s)$ are the maximum elements labeled by $\widetilde{x}_1$ and $\widetilde{x}_2$, respectively.
    \item Let $T$ be a composition of toggles weakly left of $(j,s)$, and let \[\widetilde{x} = (\widetilde{x}_1 \concat \widetilde{x}_2 \concat \widetilde{x}_3) = {\RSK^{-1}} \circ T \circ \idpro \circ \RSK(x),\] where $(i-1,s)$ and $(j-1,s)$ are the maximum elements labeled by $\widetilde{x}_1$ and $\widetilde{x}_2$, respectively.
\end{enumerate}
Then $Q(\widetilde{x}_2) = Q(x_2)$, and $\widetilde{x}_1, P(\widetilde{x}_2)$, and $\widetilde{x}_3$ are independent of $\overline{Q}(x_2)$.
\end{Lemma}

\begin{proof}
We will prove (b) (the proof of (a) is similar). Let
\[y = \RSK^{-1} \circ \ie_{j,s} \circ \RSK(x) \quad \text{and} \quad \widetilde{y} = \RSK^{-1} \circ \ie_{j-1,s} \circ \RSK(\widetilde{x}).\]
Here $y = (y_2 \concat y_1 \concat y_3)$ and $\widetilde{y} = (\widetilde{y}_2 \concat \widetilde{y}_1 \concat \widetilde{y}_3)$ where $y_i$ and $x_i$ have the same dimensions, as do $\widetilde{y}_i$ and $\widetilde{x}_i$. Let $L$ be the composition of toggles in $\idpro$ that lie weakly left of $(j,s)$, so that $\idpro = L \circ \idpro_{j,s}$. We claim that the diagram in Figure~\ref{fig:diagram} commutes. Indeed, by Observation~\ref{Obs:togglesCommute}
\begin{align*}
    {\ie_{j-1,s}} \circ T \circ \idpro &= {\ie_{j-1,s}} \circ T \circ L \circ \idpro_{j,s}\\
    &= T \circ L \circ \ie_{j-1,s} \circ \idpro_{j,s}\\
    &= T \circ L \circ \ie_{j,s}.
\end{align*}


By Proposition~\ref{Prop:evacCoordinates}, $Q(y_2) = \evac(Q(x_2))$, while the entries weakly left of $(j-i,s)$ in $\RSK(y)$ are independent of $\overline Q(x_2)$. Toggles in $T$ and $L$ are all strictly left of $(j-i,s)$, so the labels weakly left of $(j-i,s)$ in $\RSK(\widetilde y)$ are still independent of $\overline Q(x_2)$, and $Q(\widetilde y_2) = Q(y_2)$. The proof is completed by an application of Proposition~\ref{Prop:evacCoordinates} since $\widetilde{x} = \RSK^{-1} \circ \ie_{j-1,s} \circ \RSK(\widetilde{y})$.
\end{proof}

\begin{figure}
\centering
\begin{tikzpicture}[scale = 0.95]
\node at (0,0) {$\RSK(x)$};
\node at (5,0) {$\RSK(\widetilde{x})$};

\node at (0,-2) {$\RSK(y)$};
\node at (5,-2) {$\RSK(\widetilde{y})$};

\draw [->] (1,0)--(4,0);
\draw [->] (1,-2)--(4,-2);

\draw [->] (0,-0.5)--(0,-1.5);
\draw [->] (5,-0.5)--(5,-1.5);

\node at (2.5,0.3) {$T \circ \idpro$};
\node at (2.5,-1.7) {$T \circ L$};

\node at (0.85,-1) {$\ie_{j,s}$};
\node at (6.05,-1) {$\ie_{j-1,s}$};

\end{tikzpicture}
\caption{The commuting diagram in the proof of Lemma~\ref{Lemma:independenceOfQ}}
\label{fig:diagram}
\end{figure}

\begin{Prop} \label{Prop:maps}
The maps 
\[\phi \circ \rho^{-1} \circ \phi^{-1}, \, \RSK^{-1} \circ \proid \circ \RSK,\text{ and } \RSK^{-1} \circ \idpro \circ \RSK\]
have the required form for Lemma~\ref{Lemma:independenceOfQ}.
\end{Prop}

\begin{proof}
We consider $\phi \circ \rho^{-1} \circ \phi^{-1}$, which is the only nontrivial case. By Proposition~\ref{Prop:conjugate}
\[\phi \circ \rho^{-1} \circ \phi^{-1} = \RSK^{-1} \circ \ie \circ \swpro \circ \ie \circ \RSK.\]
We will express $\ie \circ \swpro \circ \ie$ as $T \circ \idpro$ for some composition $T$ of toggles weakly left of $(r,s)$. Recall that $\swpro = \left(\idpro\right)^{-1} \circ \rho_F \circ \proid$, where $F$ is the file containing $(r,s)$. Then
\begin{align*}
    \ie \circ \swpro \circ \ie &= \ie \circ \left(\idpro\right)^{-1} \circ \rho_F \circ \proid \circ \ie\\
    &={\ie_{r-1,s}} \circ \rho_F \circ {\proid} \circ (\ie_{r-1,s} \circ \idpro).
\end{align*}
By Observation~\ref{Obs:togglesCommute} the instances of $\ie_{r-1,s}$ commute with $\rho_F \circ \proid$ and cancel, yielding
\[\rho_F \circ \proid \circ \idpro = T \circ \idpro,\]
where $T = \rho_F \circ \proid$ contains only toggles weakly left of $(r,s)$.
\end{proof}

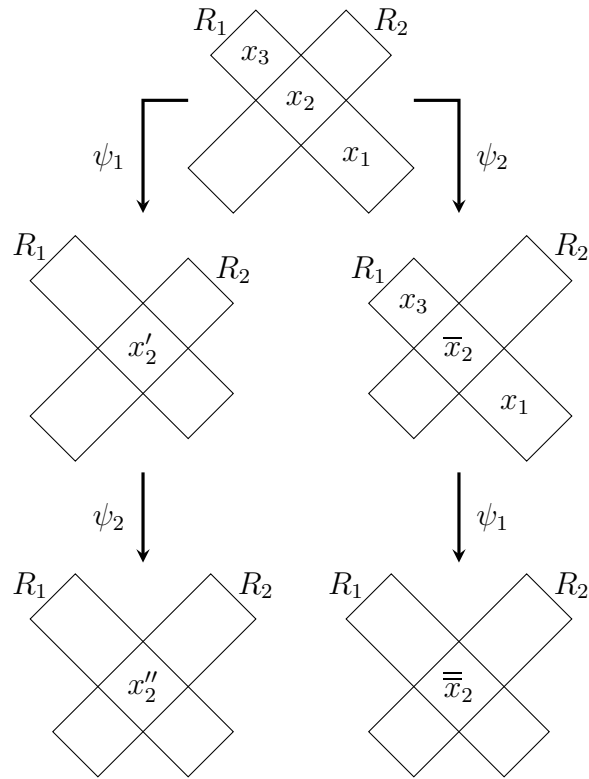
\begin{figure}
    \centering
    \begin{tikzpicture}[scale=0.3]
    
    \node at (-4,9.5) {$R_1$};
    \node at (4,9.5) {$R_2$};
    \node at (2.5,3.5) {$x_1$};
    \node at (0,6) {$x_2$};
    \node at (-2,8) {$x_3$};
    \draw (3,1)--(5,3)--(-2,10)--(-4,8)--cycle;
    \draw (-3,1)--(-5,3)--(2,10)--(4,8)--cycle;
    
    \node at (-8.5,3.5) {$\psi_1$};
    \draw [-stealth, very thick] (-5,6)--(-7,6)--(-7,1);
    
    \begin{scope}[shift={(-8,-10)}]
        \node at (-4,9.5) {$R_1$};
        \node at (5,8.5) {$R_2$};
        \node at (1,5) {$x_2'$};
        \draw (3,1)--(5,3)--(-2,10)--(-4,8)--cycle;
        \draw (-2,0)--(-4,2)--(3,9)--(5,7)--cycle;
    \end{scope}
    
    \node at (-8.5,-12.5) {$\psi_2$};
    \draw [-stealth,very thick] (-7,-10.5)--(-7,-14.5);
    
    \begin{scope}[shift={(-7,-24)}]
        \node at (-5,8.5) {$R_1$};
        \node at (5,8.5) {$R_2$};
        \node at (0,4) {$x_2''$};
        \draw (2,0)--(4,2)--(-3,9)--(-5,7)--cycle;
        \draw (-2,0)--(-4,2)--(3,9)--(5,7)--cycle;
    \end{scope}
    
    \node at (8.5,3.5) {$\psi_2$};
    \draw [-stealth, very thick] (5,6)--(7,6)--(7,1);
    
    \begin{scope}[shift={(8,-10)}]
        \node at (-5,8.5) {$R_1$};
        \node at (4,9.5) {$R_2$};
        \node at (1.5,2.5) {$x_1$};
        \node at (-1,5) {$\overline{x}_2$};
        \node at (-3,7) {$x_3$};
        \draw (2,0)--(4,2)--(-3,9)--(-5,7)--cycle;
        \draw (-3,1)--(-5,3)--(2,10)--(4,8)--cycle;
    \end{scope}
    
    \node at (8.5,-12.5) {$\psi_1$};
    \draw [-stealth,very thick] (7,-10.5)--(7,-14.5);
    
    \begin{scope}[shift={(7,-24)}]
        \node at (-5,8.5) {$R_1$};
        \node at (5,8.5) {$R_2$};
        \node at (0,4) {$\overline{\overline{x}}_2$};
        \draw (2,0)--(4,2)--(-3,9)--(-5,7)--cycle;
        \draw (-2,0)--(-4,2)--(3,9)--(5,7)--cycle;
    \end{scope}
    \end{tikzpicture}
    \caption{Labelings associated to applications of $\psi_1$ and $\psi_2$ on maximal rectangles $R_1$ and $R_2$ of a moon polyomino. Since the maps commute, $x_2'' = \overline{\overline{x}}_2$.}
    \label{fig:commutation}
\end{figure}

We are now ready to prove Theorem~\ref{Th:commutation}. The theorem is nearly immediate from Lemma~\ref{Lemma:independenceOfQ} and is similar to the proof in Rubey \cite{rubey} from this point. 

\begin{proof}[Proof of Theorem~\ref{Th:commutation}]

Consider the diagram in Figure~\ref{fig:commutation}. It suffices to consider only the moon polyomino $R_1 \cup R_2$ since all coordinates outside of this shape are preserved. By Lemma~\ref{Lemma:independenceOfQ} and Proposition~\ref{Prop:maps}, $P(x_2')$ is a function of $x_1$, $P(x_2)$ and $x_3$; and $P(\overline{\overline{x}}_2)$ is the same function of $x_1$, $P(\overline{x}_2)$, and $x_3$. But by Corollary~\ref{Cor:shiftsinonedirection}, $P(x_2) = P(\overline{x}_2)$ and so $P(x_2') = P(\overline{\overline{x}}_2)$. By Corollary~\ref{Cor:shiftsinonedirection} again, $P(x_2'') = P(x_2')$, and so $P(x_2'') = P(\overline{\overline{x}}_2)$. Similarly $Q(x_2'') = Q(\overline{\overline{x}}_2)$, so we must have $x_2''=\overline{\overline{x}}_2$. The other four sections of the moon polyomino can be treated similarly.
\end{proof}

\bibliographystyle{acm}
\bibliography{citations}

\end{document}